\documentclass[reqno]{amsart}
\usepackage{CJK}
\usepackage{hyperref}
\usepackage{cite}
\usepackage{amsmath}
\usepackage{mathrsfs}
\usepackage{xcolor}
\usepackage{bbm}
\usepackage{pifont}
\usepackage{bbding}
\usepackage{amsmath, esint}
\usepackage{mathrsfs}
\usepackage{amsfonts}
\usepackage{amssymb}
\usepackage{amsfonts,amssymb,amsmath,indentfirst,cases,amsthm}
\usepackage{epsfig}
\usepackage[numbers,sort&compress]{natbib}

\makeatletter
\@namedef{subjclassname@2020}{%
	\textup{2020} Mathematics Subject Classification}
\makeatother
\numberwithin{equation}{section}

\newtheorem{theorem}{Theorem}[section]
\newtheorem{lemma}[theorem]{Lemma}
\newtheorem{definition}[theorem]{Definition}

\newtheorem{proposition}[theorem]{Proposition}
\newtheorem{remark}[theorem]{Remark}

\allowdisplaybreaks

\begin{document}
	
\title[\hfil Mixed Local and Nonlocal equations in the Heisenberg Group]{$C^{1,\alpha}$-regularity for Mixed Local and Nonlocal Degenerate Elliptic Equations in the Heisenberg Group}

\author[J. Zhang  \hfil \hfilneg]{Junli Zhang}

\address{Junli Zhang \hfill\break School of Mathematics and Data Science, Shaanxi University of Science and Technology, Xi'an, Shaanxi, 710021, China}
\email{jlzhang2020@163.com}

\subjclass[2020]{35B65, 35D30, 35J70, 35R09, 35R11}
\keywords{$C^{1,\alpha}$-regularity, local and nonlocal operator, Heisenberg group, horizontal difference, iteration scheme of Morrey-type.}

\maketitle

\begin{abstract}
The regularity theory for equations combining both local and nonlocal operators in sub-Riemannian geometries is a huge challenge. In this paper, we investigate the $C^{1,\alpha}$-regularity of weak solutions to mixed local and nonlocal degenerate elliptic equations on the Heisenberg group. We first derive a sophisticated iteration scheme of Morrey-type by leveraging horizontal difference combined with the fractional Sobolev-type inequality on the Heisenberg group. Then, the H\"{o}lder continuity of the weak solutions is established by applying the local boundedness, the iteration scheme of Morrey-type, an iterative method and the Morrey inequality. Finally, we use the H\"{o}lder continuity in conjunction with Theorem 1.2 from Mukherjee and Zhong\cite{MZ21} to prove the $C^{1,\alpha}$-regularity of weak solutions.
\end{abstract}

\section{Introduction}
\label{sec-1}

The study of regularity theory for degenerate elliptic equations has long stood as a cornerstone of modern analysis, tracing its intellectual lineage to the groundbreaking contributions of De Giorgi, Nash, and Moser. In the context of the Heisenberg group $\mathbb{H}^n$, a paradigmatic example of a sub-Riemannian manifold, the analysis of quasilinear degenerate elliptic equations of $p$-sub-Laplacian type has undergone profound development. In the pioneering work \cite{CL97}, Capogna established $C^{1,\alpha}$-regularity for the case $p=2$. Subsequently, Mukherjee and Zhong\cite{MZ21} extended this result to quasilinear degenerate elliptic equations with $p$-growth conditions for all $1<p<\infty$ via De Giorgi's method. Further progress includes: fractional differentiability estimates for nonlinear non-differentiable degenerate elliptic equations obtained by Zhang and Niu\cite{ZN22}; ${C^{\alpha }}$-regularity of minimizers for generalized Orlicz functionals encompassing $p,q$-growth conditions established by Zhang and Niu\cite{ZN20}; weak differentiability of weak solutions under the constraint $1<p<4$ for quasilinear equations with $p,q$-growth conditions proved by Zhang and Li\cite{ZL23}. A sufficient condition excluding the Lavrentiev phenomenon for non-autonomous integral functionals with $p,q$-growth condition in the same setting provided by Zhang and Niu\cite{ZN231}. For additional developments and comprehensive references, we refer the reader to \cite{ZN22} and the works cited therein.

The regularity theory for purely nonlocal equations involving fractional powers of the $p$-Laplacian on the Heisenberg group has witnessed remarkable progress in recent years, spurred by pivotal advancements such as the introduction of nonlocal tailing mechanisms and the establishment of nonlocal Harnack and H\"{o}lder regularity estimates. These contributions, as reported in Roncal and Thangavelu\cite{RT16}, Frank, Lieb, and Seiringer\cite{FLS08}, Ferrari and Franchi\cite{FF15}, Ferrari, Miranda, Pallara, Pinamonti and Sire\cite{FMPPS18}, Frank, Gonzalez, Monticelli and Tan\cite{FGMT15}, Manfredini, Palatucci, Piccinini and Polidoro\cite{MPPP23}, and Palatucci and Piccinini\cite{PP22}, collectively constitute a transformative body of work reshaping the analytical foundations of the field. Moreover, Palatucci and Piccinini in 2022 (see \cite{PP22}) and 2023 (see \cite{PP23}) proposed an open problem how to determine the regularity of nonlocal double-phase equations in the Heisenberg group. Fang, Zhang and Zhang provided a solution in \cite{FZZ24}.

However, the analysis of equations combining local and nonlocal operators-referred to as "mixed" equations-presents a formidable challenge due to the intricate interplay between classical diffusion processes and long-range interactions. While the Euclidean setting has witnessed rapid progress, with results ranging from local boundedness (see Garain and Kinnunen\cite{GK22}) to higher-order H\"{o}lder estimates (see Garain and Lindgren\cite{GL23}), extending such theories to the non-commutative geometry of the Heisenberg group demands overcoming profound technical obstacles, including the absence of embedding theorems and the inherent anisotropy of the Carnot-Carath\'{e}odory metric.

Recently, Zhang, Niu and Wu in \cite{ZNW25} proved that both the derivative in the vertical direction and the derivative in the horizontal direction of weak solutions $u \in H{W^{1,p}}\left( \Omega  \right)$ to mixed sub-Laplace and fractional sub-Laplace equations and $f \in {L^2}\left( \Omega  \right)$ belonging to the first-order local Sobolev space $HW_{loc}^{1,p}\left( \Omega  \right)$. Furthermore, under the assumption $f \in H{W^{m,p}}\left( \Omega  \right)$, it follows that $u \in HW_{loc}^{m + 2,p}\left( \Omega  \right)$. Subsequently, Zhang and Niu\cite{ZN26} achieved a significant breakthrough by establishing fundamental regularity properties for weak solutions to mixed local and nonlocal degenerate elliptic equations in $\mathbb{H}^n$. Specifically, they demonstrated the local boundedness of weak subsolutions, the H\"{o}lder continuity of weak solutions, and critically, both the Harnack inequality and the weak Harnack inequality. These results affirm that the classical difference and De Giorgi-Nash-Moser theory can be fruitfully extended to this hybrid setting, laying a robust foundation for the qualitative analysis of such equations. While H\"{o}lder continuity of the solution $u$ has been established, a central open question remains: does the horizontal gradient ${{\nabla _H}u}$ exhibit H\"{o}lder continuity, i.e., do solutions belong to the $C^{1,\alpha}$ class? This problem is exceptionally intricate in the context of degenerate ($p>2$) equations coupled with mixed operators, as standard linearization techniques prove inadequate, and the interplay between the local $p$-Laplacian structure and the nonlocal fractional term obfuscates the differentiation of the equation.

In this paper, we bridge this gap by investigating the $C^{1,\alpha}$-regularity of weak solutions to mixed local and nonlocal degenerate elliptic equations in the Heisenberg group. Building on the Harnack estimates and tail control developed in \cite{ZN26}, we aim to prove that under suitably structured assumptions, weak solutions not only exhibit H\"{o}lder continuity but also possess higher smoothness properties with their horizontal derivatives satisfying a H\"{o}lder condition. This result will extend the classical $C^{1,\alpha}$ estimates for $p$-sub-Laplace equations in $\mathbb{H}^n$ to the more intricate landscape of mixed local-nonlocal diffusion.

In this paper we focus on the analysis of the mixed local and nonlocal degenerate elliptic equations
\begin{equation}\label{eq0}
 - \Delta _{{{\mathbb{H}},p}}u + \Lambda{\left( { - \Delta _{{{\mathbb{H}},p}}} \right)^s}u = 0\;{\rm{in}}\;\Omega ,
\end{equation}
where $2\le p< \infty,\;0<s<1,\;0\le\Lambda \le 1 $, and $\Omega$ is a bounded open subset in the Heisenberg group  ${{\mathbb{H}}^n},\;n\ge1$. Here the $p$-sub-Laplace operator on the Heisenberg group is defined by
\begin{equation}\label{eq12}
 - {\Delta _{{\mathbb{H}},p}}u = di{v_H}\left( {{{\left| {{\nabla _H}u} \right|}^{p - 2}}{\nabla _H}u} \right),
\end{equation}
while the fractional $p$-sub-Laplace operator on the Heisenberg group is given by
\begin{equation}\label{eq13}
{\left( { - {\Delta _{{\mathbb{H}},p}}} \right)^s}u = {\rm{P}}{\rm{.V}}{\rm{.}}\int_{{{\mathbb{H}}^n}} {\frac{{{{\left| {u\left( \xi  \right) - u\left( \eta  \right)} \right|}^{p - 2}}\left( {u\left( \xi  \right) - u\left( \eta  \right)} \right)}}{{{{\left\| {{\eta ^{ - 1}} \circ \xi } \right\|}^{Q + sp}_{\mathbb{H}^n}}}}d\eta },
\end{equation}
where P.V. signifies the Cauchy principal value and $Q=2n+2$ is the homogeneous dimension of ${{\mathbb{H}}^n}$.

The main results are as follows:
\begin{theorem}\label{Th13}
Let $2\le p <\infty,\;0<s<1$ and $0\le \Lambda \le 1$. Suppose $\Omega \in \mathbb{H}^n$ is a bounded open set and $u \in HW_{{\rm{loc}}}^{1,p}\left( \Omega \right) \cap{ L_{sp}^{p - 1}\left( {{\mathbb{H}^n}} \right)} $ is a weak solution of \eqref{eq0}. Then $u \in C_{loc}^\gamma \left( \Omega  \right)$ for any $0 < \gamma  < 1$. Moreover, for any $0 < \gamma  < 1$ and any Kor\'{a}nyi ball ${B_{2R}}\left( {{\xi _0}} \right) \subset  \subset \Omega $ with $0<R<1$, there exists a positive constant $c=c(Q,p,s,\gamma)$ such that
\begin{equation}\label{eq482}
  \mathop {\sup }\limits_{\xi  \ne \eta  \in {B_{\frac{R}{2}}}\left( {{\xi _0}} \right)} \frac{{\left| {u\left( \xi  \right) - u\left( \eta  \right)} \right|}}{{{{\left\| {{\eta ^{ - 1}} \circ \xi } \right\|}^\gamma _{\mathbb{H}^n}}}} \le \frac{c}{{{R^\gamma }}}\left( {{{{\left\| u \right\|}_{{L^\infty }\left( {{B_R}\left( {{\xi _0}} \right)} \right)}}} + {\rm{Tai}}{{\rm{l}}_{p - 1,sp,p}}\left( {u;{\xi _0},R} \right)} \right),
\end{equation}
where the definition of ${\rm{Tai}}{{\rm{l}}_{p - 1,sp,p}}\left( {u;{\xi _0},R} \right)$ see \eqref{eq26} below.
\end{theorem}

\begin{remark}
In contrast to \cite{ZN26}, where the H\"older index $\gamma \in (0,\frac{sp}{p-1})$ relies on the parameters $s$ and $p$, Theorem \ref{Th13} establishes a parameter-independent H\"older regularity with $\gamma \in (0,1)$. This holds under more general nonlinear frameworks and is particularly advantageous for subsequent regularity analysis.
\end{remark}

\begin{theorem}\label{Th14}
Let $2\le p <\infty,\;0<s<1,\;sp<p-1$ and $0\le \Lambda \le 1$. Suppose $\Omega \in \mathbb{H}^n$ is a bounded open set and $u \in HW_{{\rm{loc}}}^{1,p}\left( \Omega \right) \cap{ L_{sp}^{p - 1}\left( {{\mathbb{H}^n}} \right)} $ is a weak solution of \eqref{eq0}. Then $u \in C_{loc}^{1,\alpha} \left( \Omega  \right)$ for some $0 < \alpha  < 1$. Moreover, for some $0 < \alpha  < 1$ and any Kor\'{a}nyi ball ${B_{2R}}\left( {{\xi _0}} \right) \subset  \subset \Omega $ with $0<R<1$, there exists a positive constant $c=c(Q,p,q,s,\gamma)$ such that
\begin{equation}\label{eq51}
 \mathop {\sup }\limits_{\xi  \ne \eta  \in {B_{\frac{R}{8}}}\left( {{\xi _0}} \right)} \frac{{\left| {{\nabla _H}u\left( \xi  \right) - {\nabla _H}u\left( \eta  \right)} \right|}}{{{{\left\| {{\eta ^{ - 1}} \circ \xi } \right\|}^\alpha _{\mathbb{H}^n}}}} \le \frac{c}{{{R^{1 + \alpha }}}}\left( {{{\left\| u \right\|}_{{L^\infty }\left( {{B_R}\left( {{\xi _0}} \right)} \right)}} + {\rm{Tai}}{{\rm{l}}_{p - 1,sp,p}}\left( {u;{\xi _0},R} \right)} \right),
\end{equation}
where the definition of ${\rm{Tai}}{{\rm{l}}_{p - 1,sp,p}}\left( {u;{\xi _0},R} \right)$ see \eqref{eq26} below.
\end{theorem}

The paper is organized as follows: In Section 2, we introduce relevant knowledge of the Heisenberg group, function spaces, and several necessary and new lemmas. In Section 3, we employ horizontal difference, along with certain difference estimates and the fractional Sobolev-type inequality on the Heisenberg group, to establish the crucial iteration scheme of Morrey-type for weak solutions to \eqref{eq0} required for the proof of the Theorem \ref{Th13}, i.e. Proposition \ref{Pro41}. In Section 4, we first establish a proposition regarding the gradient integrability of weak solutions by using Caccioppoli-type inequality, i.e. Proposition \ref{Pro42}. Then, by combining local boundedness, the Proposition \ref{Pro41}, the Proposition \ref{Pro42}, an iterative method, and the Morrey inequality, we prove Theorem \ref{Th13}. Finally, Theorem \ref{Th14} is derived by utilizing Theorem \ref{Th13} and Theorem 1.2 from \cite{MZ21}.

In the sequel, we denote by $c$ a generic positive constant which may vary from line to line. Relevant dependencies on parameters shall be emphasised utilizing parentheses, i.e., $c= c(n, p, q)$ means that $c$ depends on $n, p, q$.

\section{Preliminaries}
\label{Section 2}

In this section, we introduce briefly some relevant knowledge of the Heisenberg group ${{\mathbb{H}}^n}$, some function spaces and several necessary lemmas.

\subsection{The Heisenberg group $\mathbb{H}^n$}\label{Section 21}
The Euclidean space ${\mathbb{R}^{2n + 1}},\;n \ge 1$ with the group multiplication
\begin{equation}\label{eq21}
\xi \circ \eta = \left( {{x_1} + {y_1},{x_2} + {y_2}, \cdots ,{x_{2n}} + {y_{2n}},t + \tau + \frac{1}{2}\sum\limits_{i = 1}^n {\left( {{x_i}{y_{n + i}} - {x_{n + i}}{y_i}} \right)} } \right),
\end{equation}
where $\xi = \left( {{x_1},{x_2}, \cdots ,{x_{2n}},t} \right),\;\eta = \left( {{y_1},{y_2}, \cdots ,{y_{2n}},\tau} \right) \in {\mathbb{R}^{2n + 1}},$ leads to the Heisenberg group $\mathbb{H}^n$. The left invariant vector fields on $\mathbb{H}^n$ are of the form
\begin{equation}\label{eq22}
{X_i} = {\partial _{{x_i}}} - \frac{{{x_{n + i}}}}{2}{\partial _t},\;{X_{n + i}} = {\partial _{{x_{n + i}}}} + \frac{{{x_i}}}{2}{\partial _t},\;\;1 \le i \le n,
\end{equation}
and a non-trivial commutator on $\mathbb{H}^n$ is
\[T = {\partial _t} = \left[ {{X_i},{X_{n + i}}} \right] = {X_i}{X_{n + i}} - {X_{n + i}}{X_i},\;\;1 \le i \le n.\]
We call that ${X_1},{X_2}, \cdots ,{X_{2n}}$ are the horizontal vector fields and $T$ the vertical vector field on $\mathbb{H}^n$. Denote the horizontal gradient of a smooth function $u$ on $\mathbb{H}^n$ by
\[{\nabla _H}u = \left( {{X_1}u,{X_2}u, \cdots ,{X_{2n}}u} \right).\]
Denote all possible $k$-th order partial derivatives of $u$ by $D^k_H u$. The Haar measure in $\mathbb{H}^n$ is equivalent to the Lebesgue measure in ${\mathbb{R}^{2n + 1}}$. The Lebesgue measure of a measurable set $E \subset {\mathbb{H}^n}$ is denoted by $\left| E \right|.$

The Carnot-Carath\`{e}odary metric (C-C metric) between two points in $\mathbb{H}^n$ is the shortest length of the horizontal curve joining them, denoted by $d$. The Kor\'{a}nyi ball induced by the C-C metric is
\[{B_R}\left( \xi \right) = \left\{ {\eta \in {\mathbb{H}^n}:d\left( {\eta,\xi} \right) < R} \right\}.\]
The homogeneous dimension of ${{\mathbb{H}}^n}$ is $Q=2n+2$. For $\xi = \left( {{x_1},{x_2}, \cdots ,{x_{2n}},t} \right)\in \mathbb{H}^n,$ its module is defined as
\[{\left\| \xi \right\|_{{\mathbb{H}^n}}} = {\left( {{{\left( {\sum\limits_{i = 1}^{2n} {{x_i}^2} } \right)}^2} + {t^2}} \right)^{\frac{1}{4}}}.\]
The C-C metric $d$ is equivalent to the Kor\`{a}nyi metric
\[d\left( {\xi,\eta} \right) \sim {\left\| {{\xi^{ - 1}} \circ \eta} \right\|_{{\mathbb{H}^n}}}.\]

If $Z$ is a left invariant vector field on $\mathbb{H}^n$, then for some $z = \left( {{z_1},{z_2}, \cdots ,{z_{2n + 1}}} \right) = \left( {z',{z_{2n + 1}}} \right)\in \mathbb{H}^n,$  we write
\[Z = \sum\limits_{l = 1}^{2n} {{z_l}{X_l}}  + {z_{2n + 1}}T.\]
The exponential mapping in canonical coordinates is defined as
\[{{\mathop{\rm e}\nolimits} ^Z} = z.\]
By \eqref{eq21}, it follows the Baker-Campbell-Hausdorff formula: if $Z$ and $Y$ are left invariant vector fields with components $z$ and $y$, then
\begin{equation}\label{eq23}
{{\mathop{\rm e}\nolimits} ^Z}{{\mathop{\rm e}\nolimits} ^Y} = \left( {z',{z_{2n + 1}}} \right) \circ \left( {y',{y_{2n + 1}}} \right) = {{\mathop{\rm e}\nolimits} ^{Z + Y + \frac{1}{2}\left[ {Z,Y} \right]}}.
\end{equation}

For $h \in \mathbb{R}\backslash \left\{ 0 \right\}$, the first order Nirenberg difference of the function $v$ along the left invariant vector field $Z$-direction is defined as
\begin{equation}\label{eq24}
{\Delta _{Z,h}}v\left( x \right) = v\left( {x{e^{hZ}}} \right) - v\left( x \right),
\end{equation}
and the second order Nirenberg difference is defined as
\begin{equation}\label{eq25}
 \Delta _{Z,h}^2v\left( x \right) = v\left( {x{e^{2hZ}}} \right) -2v\left( {x{e^{ hZ}}} \right) +v\left( x \right).
\end{equation}
For more details, see \cite{CL97} and \cite{DA04}.

\subsection{Function spaces on $\mathbb{H}^n$}\label{Section 22}
For $0<\alpha<1$ and $\Omega  \subset {\mathbb{H}^n},$ the H\"{o}lder spaces, also known as the Folland-Stein classes (see \cite{FS74} and \cite{FS82}), are defined as
\[{C^\alpha }\left( \Omega  \right) = \left\{ {u\in{C }\left( \Omega  \right) |\mathop {\sup }\limits_{\xi \neq \eta  \in \Omega } \frac{{\left| {u\left( \xi \right) - u\left( \eta \right)} \right|}}{{\left\| {{\eta ^{ - 1}} \circ \xi } \right\|_{{{\mathbb{H}}^n}}^\alpha }} < \infty } \right\},\]
with their local spaces given by
\[C _{loc}^\alpha \left( \Omega  \right) = \left\{ {u|\varphi u \in {C ^\alpha }\left( \Omega  \right),\;\;\varphi  \in C_0^\infty \left( \Omega  \right)} \right\}.\]
If $u,{D _H}u, \cdots ,D _H^ku \in {C ^\alpha }\left( \Omega  \right), \;k\in \mathbb{N}^+$, then we say $u \in {C ^{k,\alpha }}\left( \Omega  \right)$. It is worth noting that ${C ^{k,\alpha }}\left( \Omega  \right)$ forms a Banach space under the norm
\[{\left\| u \right\|_{{C^{k,\alpha }}\left( \Omega  \right)}} = \sum\limits_{ k}\mathop {\sup }\limits_{\xi \neq \eta  \in \Omega } \frac{{\left| {D _H^ku\left( \xi  \right) - D _H^ku\left( \eta  \right)} \right|}}{{\left\| {{\eta ^{ - 1}} \circ \xi } \right\|_{{{\mathbb{H}}^n}}^\alpha }} + \sum\limits_{0 \le m \le k} {{{\left\| {D _H^mu} \right\|}_{C\left( \Omega  \right)}}} .\]

For $1 \le p < \infty $, $k \in {\mathbb{N}^ + }$ and $\Omega  \subset {\mathbb{H}^n},$ the horizontal Sobolev space $H{W^{k,p}}\left( \Omega  \right)$ is defined as
\[H{W^{k,p}}\left( \Omega  \right) = \left\{ {u \in {L^p}\left( \Omega  \right):{\nabla _H}u \in {L^p}\left( \Omega  \right),\nabla _H^2u \in {L^p}\left( \Omega  \right), \cdots ,\nabla _H^ku \in {L^p}\left( \Omega  \right)} \right\},\]
which is a Banach space under the norm
\[{\left\| u \right\|_{H{W^{k,p}}\left( \Omega  \right)}} = {\left\| u \right\|_{{L^p}\left( \Omega  \right)}} + \sum\limits_{m = 1}^k {{{\left\| {\nabla _H^mu} \right\|}_{{L^p}\left( \Omega  \right)}}} .\]
The local horizontal Sobolev space $HW_{loc}^{k,p}\left( \Omega  \right)$ is defined as
\[HW_{loc}^{k,p}\left( \Omega  \right): = \left\{ {u:u \in H{W^{k,p}}\left( {\Omega '} \right),\forall \Omega ' \subset  \subset \Omega } \right\}\]
and the space $HW_0^{k,p}\left( \Omega  \right)$ is the closure of $C_0^\infty \left( \Omega  \right)$ in $H{W^{k,p}}\left( \Omega  \right)$.

For $1 \le p < \infty $, $s \in \left( {0,1} \right)$, the Gagliardo semi-norm of $u$ on ${{\mathbb{H}}^n}$ is defined as
\[{\left[ u \right]_{H{W^{s,p}}\left( {{{\mathbb{H}}^n}} \right)}} = {\left( {\int_{{{\mathbb{H}}^n}} {\int_{{{\mathbb{H}}^n}} {\frac{{{{\left| {u\left( \xi  \right) - u\left(\eta  \right)} \right|}^p}}}{{\| {{\eta^{ - 1}} \circ \xi } \|_{{{\mathbb{H}}^n}}^{Q + sp}}}\,d\xi } d\eta } } \right)^{\frac{1}{p}}},\]
and the fractional Sobolev spaces $H{W^{s,p}}\left( {{{\mathbb{H}}^n}} \right)$ on the Heisenberg group are defined as
\[H{W^{s,p}}\left( {{{\mathbb{H}}^n}} \right) = \left\{ {u \in {L^p}\left( {{{\mathbb{H}}^n}} \right):{{\left[ u \right]}_{H{W^{s,p}}\left( {{{\mathbb{H}}^n}} \right)}} < \infty } \right\}\]
endowed with the natural fractional norm
\[{\| u \|_{H{W^{s,p}}\left( {{{\mathbb{H}}^n}} \right)}} = {\left( {\| u \|_{{L^p}\left( {{{\mathbb{H}}^n}} \right)}^p + \left[ u \right]_{H{W^{s,p}}\left( {{{\mathbb{H}}^n}} \right)}^p} \right)^{\frac{1}{p}}}.\]

We define the tail space
\[L_\alpha ^q\left( {{{\mathbb{H}}^n}} \right) = \left\{ {u \in L_{loc}^q\left( {{{\mathbb{H}}^n}} \right):\int_{{{\mathbb{H}}^n}} {\frac{{{{\left| u \right|}^q}}}{{1 + {{\left\| \xi  \right\|}^{Q + \alpha }_{\mathbb{H}^n}}}}d\xi }  < \infty } \right\},\;q > 0\;{\rm{and}}\;\alpha  > 0.\]
For any $\xi_0 \in\mathbb{H}^n$, $R>0$, $\beta>0$ and $u \in L_\alpha^q(\mathbb{H}^n)$, we define
\begin{equation}\label{eq26}
 {\rm{Tai}}{{\rm{l}}_{q,\alpha ,\beta }}\left( {u;{\xi _0},R} \right) = {\left[ {{R^\beta }\int_{{{\mathbb{H}}^n}\backslash {B_R}\left( {{\xi _0}} \right)} {\frac{{{{\left| u \right|}^q}}}{{{{\left\| {\xi _0^{ - 1} \circ \xi } \right\|}^{Q + \alpha }_{\mathbb{H}^n}}}}d\xi } } \right]^{\frac{1}{q}}}.
\end{equation}

\subsection{Several necessary lemmas}\label{Section 23}

\begin{lemma}[\cite{CL97}]\label{Le22}
Let $\Omega  \subset {\mathbb{H}^n}$ be an open set, $K \subset \Omega $ a
compact set, $Z$ a left invariant vector field and $v \in L_{loc}^p\left( \Omega  \right)$ for $1 \le p < \infty $. If there
exist the positive constants $\tilde h$ and $c$ such that
\[\mathop {\sup }\limits_{0 < \left| h \right| < \tilde h} \int_K {{{\left|  \frac{{v\left( {x{e^{hZ}}} \right) - v\left( x \right)}}{{{{\left| h \right|}}}} \right|}^p}d\xi}  \le {c^p},\]
then
\[Zv \in {L^p}\left( K \right)\; \hbox{and}\;{\left\| {Zv} \right\|_{{L^p}\left( K \right)}} \le c.\]
Conversely, if $Zv \in {L^p}\left( K \right)$, then for some $\tilde h > 0$,
\[\mathop {\sup }\limits_{0 < \left| h \right| < \tilde h} \int_K {{{\left|  \frac{{v\left( {x{e^{hZ}}} \right) - v\left( x \right)}}{{{{\left| h \right|}}}} \right|}^p}d\xi}  \le {\left( {2{{\left\| {Zv} \right\|}_{{L^p}\left( K \right)}}} \right)^p}.\]
\end{lemma}

\begin{lemma}[Morrey inequality (\cite{F75}, Theorem 5.15)]\label{Le23}
Let $1<p<\infty$, $k>\frac{Q}{p}$. Then for any $u \in HW^{k,p} ({{\mathbb{H}}^n})$, we have
\[{\left\| u \right\|_{{C^{k - \frac{Q}{p}}}}({{\mathbb{H}}^n})} \le c{\left\| u \right\|_{H{W^{k,p}}}({{\mathbb{H}}^n})},\]
where $c=c(Q,p)>0.$
\end{lemma}

%xuyaogai
\begin{definition}\label{De26}
Let $u \in HW_{loc}^{1,p}\left( \Omega  \right)\cap{L_{sp} ^{p-1}\left( {{{\mathbb{H}}^n}} \right)}$. The function $u$ is called a weak subsolution (supersolution) to \eqref{eq0} if for any $\Omega ' \subset  \subset \Omega $ and non-negative test function $\phi  \in HW_0^{1,p}\left( {\Omega '} \right)$, the inequality
\begin{align}\label{eq22}
   & \int_{\Omega '} {{{\left| {{\nabla _H}u} \right|}^{p - 2}}{\nabla _H}u \cdot {\nabla _H}\phi d\xi } \nonumber \\
  +& \Lambda \int_{{{\mathbb{H}}^n}} {\int_{{{\mathbb{H}}^n}} {\frac{{{{\left| {u\left( \xi  \right) - u\left( \eta  \right)} \right|}^{p - 2}}\left( {u\left( \xi  \right) - u\left( \eta  \right)} \right)\left( {\phi \left( \xi  \right) - \phi \left( \eta  \right)} \right)}}{{\left\| {{\eta ^{ - 1}} \circ \xi } \right\|_{{{\mathbb{H}}^n}}^{Q + sp}}}d\xi d\eta } }  \le \left(  \ge  \right)0
\end{align}
holds. If $u$ is both a weak subsolution and a weak supersolution to \eqref{eq0}, then $u$ is a weak solution to \eqref{eq0}.
\end{definition}

For convenience, we denote
\[{J_p}\left( a \right) = {\left| a \right|^{p - 1}}a,\;a \in \mathbb{R},\;\;d\mu  = \frac{{d\xi d\eta }}{{\left\| {{\eta ^{ - 1}} \circ \xi } \right\|_{{{\mathbb{H}}^n}}^{Q + sp}}}.\]

%Le26来自于\cite{BLS18}引理2.2的推广
\begin{lemma}\label{Le29}
Let $\alpha>0$ and $0<q<\infty$. For every $0<r<R$ and $\xi_0 \in \mathbb{H}^n$, we have
\[\mathop {\sup }\limits_{\xi  \in {B_r}\left( {{\xi _0}} \right)} {\int _{{{\mathbb{H}}^n}\backslash {B_R}\left( {{\xi _0}} \right)}}\frac{{{{\left| {u\left( \eta  \right)} \right|}^{q}}}}{{\left\| {{\eta ^{ - 1}} \circ \xi } \right\|_{{{\mathbb{H}}^n}}^{Q + \alpha }}}d\eta  \le {\left( {\frac{R}{{R - r}}} \right)^{Q + \alpha }}{\int _{{{\mathbb{H}}^n}\backslash {B_R}\left( {{\xi _0}} \right)}}\frac{{{{\left| {u\left( \eta  \right)} \right|}^{q}}}}{{\left\| {{\eta ^{ - 1}} \circ {\xi _0}} \right\|_{{{\mathbb{H}}^n}}^{Q + \alpha}}}d\eta .\]
\end{lemma}

\begin{proof}
For ${\xi  \in {B_r}\left( {{\xi _0}} \right)}$ and $\eta \in {{{\mathbb{H}}^n}\backslash {B_R}\left( {{\xi _0}} \right)}$, we derive
\begin{align*}
   {\left\| {{\eta ^{ - 1}} \circ \xi } \right\|_{{{\mathbb{H}}^n}}}& \ge {\left\| {{\eta ^{ - 1}} \circ {\xi _0}} \right\|_{{{\mathbb{H}}^n}}} - {\left\| {\xi _0^{ - 1} \circ \xi } \right\|_{{{\mathbb{H}}^n}}} \ge {\left\| {{\eta ^{ - 1}} \circ {\xi _0}} \right\|_{{{\mathbb{H}}^n}}} - r \\
   &  \ge {\left\| {{\eta ^{ - 1}} \circ {\xi _0}} \right\|_{{{\mathbb{H}}^n}}} - \frac{r}{R}{\left\| {{\eta ^{ - 1}} \circ {\xi _0}} \right\|_{{{\mathbb{H}}^n}}} = \frac{{R - r}}{R}{\left\| {{\eta ^{ - 1}} \circ {\xi _0}} \right\|_{{{\mathbb{H}}^n}}}.
\end{align*}
This completes the proof.
\end{proof}

%Le26来自于\cite{BLS18}引理2.3的推广
\begin{lemma}\label{Le210}
Let $\alpha>0$ and $0<q<\infty$. Assume ${B_r}\left( {{\xi _0}} \right) \subset {B_R}\left( {{\xi _1}} \right)$. Then for every $u \in L_\alpha ^q\left( {{{\mathbb{H}}^n}} \right) $, we have
\begin{align*}
   & {\int _{{{\mathbb{H}}^n}\backslash {B_r}\left( {{\xi _0}} \right)}}\frac{{{{\left| {u\left( \eta  \right)} \right|}^q}}}{{\left\| {{\eta ^{ - 1}} \circ {\xi _0}} \right\|_{{{\mathbb{H}}^n}}^{Q + \alpha }}}d\eta  \\
   \le&  {\left( {\frac{R}{{R - {{\left\| {\xi _1^{ - 1} \circ {\xi _0}} \right\|}_{{{\mathbb{H}}^n}}}}}} \right)^{Q + \alpha }}{\int _{{{\mathbb{H}}^n}\backslash {B_R}\left( {{\xi _1}} \right)}}\frac{{{{\left| {u\left( \eta  \right)} \right|}^q}}}{{\left\| {{\eta ^{ - 1}} \circ {\xi _1}} \right\|_{{{\mathbb{H}}^n}}^{Q + \alpha }}}d\eta  + {r^{ - Q-\alpha}}\left\| u \right\|_{{L^q}\left( {{B_R}\left( {{\xi _1}} \right)} \right)}^q.
\end{align*}
If in addition ${u \in L_{loc}^m\left( {{{\mathbb{H}}^n}} \right)}$ for some $q<m \le \infty$, then
\begin{align*}
   {\int _{{{\mathbb{H}}^n}\backslash {B_r}\left( {{\xi _0}} \right)}}\frac{{{{\left| {u\left( \eta  \right)} \right|}^q}}}{{\left\| {{\eta ^{ - 1}} \circ {\xi _0}} \right\|_{{{\mathbb{H}}^n}}^{Q + \alpha }}}d\eta  \le& {\left( {\frac{R}{{R - {{\left\| {\xi _0^{ - 1} \circ {\xi _1}} \right\|}_{{{\mathbb{H}}^n}}}}}} \right)^{Q + \alpha }}{\int _{{{\mathbb{H}}^n}\backslash {B_R}\left( {{\xi _1}} \right)}}\frac{{{{\left| {u\left( \eta  \right)} \right|}^q}}}{{\left\| {{\eta ^{ - 1}} \circ {\xi _1}} \right\|_{{{\mathbb{H}}^n}}^{Q + \alpha }}}d\eta  \\
   &  + c\left( {Q,q,m,\alpha } \right){r^{ - \frac{{\alpha m + Qq}}{m}}}\left\| u \right\|_{{L^m}\left( {{B_R}\left( {{\xi _1}} \right)} \right)}^q.
\end{align*}
\end{lemma}

\begin{proof}
Note that
\[{\left\| {{\eta ^{ - 1}} \circ {\xi _0}} \right\|_{{{\mathbb{H}}^n}}} \ge {\left\| {{\eta ^{ - 1}} \circ {\xi _1}} \right\|_{{{\mathbb{H}}^n}}} - {\left\| {\xi _1^{ - 1} \circ {\xi _0}} \right\|_{{{\mathbb{H}}^n}}} \ge \frac{{R - {{\left\| {\xi _1^{ - 1} \circ {\xi _0}} \right\|}_{{{\mathbb{H}}^n}}}}}{R}{\left\| {{\eta ^{ - 1}} \circ {\xi _1}} \right\|_{{{\mathbb{H}}^n}}}, \;\rm{for}\; \rm{every}\; \eta  \notin {B_R}\left( {{\xi _1}} \right),\]
and
\[{\left\| {{\eta ^{ - 1}} \circ {\xi _0}} \right\|_{{{\mathbb{H}}^n}}} \ge r, \;\rm{for}\; \rm{every}\; \eta  \notin {B_r}\left( {{\xi _0}} \right),\]
we derive
\begin{align*}
   &{\int _{{{\mathbb{H}}^n}\backslash {B_r}\left( {{\xi _0}} \right)}}\frac{{{{\left| {u\left( \eta  \right)} \right|}^q}}}{{\left\| {{\eta ^{ - 1}} \circ {\xi _0}} \right\|_{{{\mathbb{H}}^n}}^{Q + \alpha }}}d\eta = {\int _{{{\mathbb{H}}^n}\backslash {B_R}\left( {{\xi _1}} \right)}}\frac{{{{\left| {u\left( \eta  \right)} \right|}^q}}}{{\left\| {{\eta ^{ - 1}} \circ {\xi _0}} \right\|_{{{\mathbb{H}}^n}}^{Q + \alpha }}}d\eta  + {\int _{{B_R}\left( {{\xi _1}} \right)\backslash {B_r}\left( {{\xi _0}} \right)}}\frac{{{{\left| {u\left( \eta  \right)} \right|}^q}}}{{\left\| {{\eta ^{ - 1}} \circ {\xi _0}} \right\|_{{{\mathbb{H}}^n}}^{Q + \alpha }}}d\eta  \\
   \le&  {\left( {\frac{R}{{R - {{\left\| {\xi _0^{ - 1} \circ {\xi _1}} \right\|}_{{{\mathbb{H}}^n}}}}}} \right)^{Q + \alpha }}{\int _{{{\mathbb{H}}^n}\backslash {B_R}\left( {{\xi _1}} \right)}}\frac{{{{\left| {u\left( \eta  \right)} \right|}^q}}}{{\left\| {{\eta ^{ - 1}} \circ {\xi _1}} \right\|_{{{\mathbb{H}}^n}}^{Q + \alpha }}}d\eta  + {r^{ - Q - \alpha }}\left\| u \right\|_{{L^q}\left( {{B_R}\left( {{\xi _1}} \right)} \right)}^q.
\end{align*}

In order to obtain the second inequality, we use H\"{o}lder's inequality to get
\begin{align*}
   {\int _{{B_R}\left( {{\xi _1}} \right)\backslash {B_r}\left( {{\xi _0}} \right)}}\frac{{{{\left| {u\left( \eta  \right)} \right|}^q}}}{{\left\| {{\eta ^{ - 1}} \circ {\xi _0}} \right\|_{{{\mathbb{H}}^n}}^{Q + \alpha }}}d\eta& \le \left\| u \right\|_{{L^m}\left( {{B_R}\left( {{\xi _1}} \right)} \right)}^q{\left( {{\int _{{{\mathbb{H}}^n}\backslash {B_r}\left( {{\xi _0}} \right)}}\frac{1}{{\left\| {{\eta ^{ - 1}} \circ {\xi _0}} \right\|_{{{\mathbb{H}}^n}}^{\left( {Q + \alpha } \right)\frac{m}{{m - q}}}}}d\eta } \right)^{\frac{{m - q}}{m}}} \\
   &  = c\left( {Q,q,m,\alpha } \right){r^{ - \frac{{\alpha m + Qq}}{m}}}\left\| u \right\|_{{L^m}\left( {{B_R}\left( {{\xi _1}} \right)} \right)}^q.
\end{align*}
Thus the proof is complete.
\end{proof}

%Le28来自于\cite{BL17}引理2.3的推广
\begin{lemma}\label{Le28}
Let $1 \le q< \infty$ and $0<\alpha< 1 $. If $u \in L^q\left( {{{\mathbb{H}}^n}} \right)$ satisfies
\[\mathop {\sup }\limits_{0 < \left| h \right|} {\left\| {\frac{{\Delta _{Z,h}^2 u}}{{{{\left| h \right|}^\alpha }}}} \right\|_{{L^q}\left( {{{\mathbb{H}}^n}} \right)}^q} < \infty, \]
then it holds
\begin{equation}\label{eq461}
\mathop {\sup }\limits_{0 < \left| h \right|} \left\| {\frac{{{\Delta _{Z,h}}u}}{{{{\left| h \right|}^\alpha }}}} \right\|_{{L^q}\left( {{{\mathbb{H}}^n}} \right)} \le \frac{c}{{1 - \alpha }}\left[ {\mathop {\sup }\limits_{0 < \left| h \right|} \left\| {\frac{{\Delta _{Z,h}^2u}}{{{{\left| h \right|}^\alpha }}}} \right\|_{{L^q}\left( {{{\mathbb{H}}^n}} \right)} + \left\| u \right\|_{{L^q}\left( {{{\mathbb{H}}^n}} \right)}} \right],
\end{equation}
for some constant $c>0$. For any $h_0 >0$, we also have
\begin{equation}\label{eq462}
\mathop {\sup }\limits_{0 < \left| h \right| < {h_0}} \left\| {\frac{{{\Delta _{Z,h}}u}}{{{{\left| h \right|}^\alpha }}}} \right\|_{{L^q}\left( {{{\mathbb{H}}^n}} \right)} \le \frac{c}{{1 - \alpha }}\left[ {\mathop {\sup }\limits_{0 < \left| h \right| < {h_0}} \left\| {\frac{{\Delta _{Z,h}^2u}}{{{{\left| h \right|}^\alpha }}}} \right\|_{{L^q}\left( {{{\mathbb{H}}^n}} \right)} + \left( {h_0^{ - \alpha } + 1} \right)\left\| u \right\|_{{L^q}\left( {{{\mathbb{H}}^n}} \right)}} \right].
\end{equation}
\end{lemma}

\begin{proof}
Using
\[{\Delta _{Z,h}}u\left( \xi  \right) = \frac{1}{2}\left( {{\Delta _{Z,2h}}u\left( \xi  \right) - \Delta _{Z,h}^2u\left( \xi  \right)} \right),\]
we obtain
\begin{equation}\label{eq460}
{\left\| {\frac{{{\Delta _{Z,h}}u}}{{{{\left| h \right|}^\alpha }}}} \right\|_{{L^q}\left( {{{\mathbb{H}}^n}} \right)}} \le \frac{1}{2}{\left\| {\frac{{{\Delta _{Z,2h}}u}}{{{{\left| h \right|}^\alpha }}}} \right\|_{{L^q}\left( {{{\mathbb{H}}^n}} \right)}} + \frac{1}{2}{\left\| {\frac{{\Delta _{Z,h}^2u}}{{{{\left| h \right|}^\alpha }}}} \right\|_{{L^q}\left( {{{\mathbb{H}}^n}} \right)}}
\end{equation}
for any $h \in \mathbb{R} \backslash \{ 0\} $. For the first term on the right-hand side of \eqref{eq460}, we let $h'=2h$ and have
\begin{align*}
   &{\left\| {\frac{{{\Delta _{Z,2h}}u}}{{{{\left| h \right|}^\alpha }}}} \right\|_{{L^q}\left( {{{\mathbb{H}}^n}} \right)}} = {2^\alpha }{\left\| {\frac{{{\Delta _{Z,h'}}u}}{{{{\left| {h'} \right|}^\alpha }}}} \right\|_{{L^q}\left( {{{\mathbb{H}}^n}} \right)}} \\
   &  \le {2^\alpha }\mathop {\sup }\limits_{0 < \left| {h'} \right| < \frac{1}{2}} {\left\| {\frac{{{\Delta _{Z,h'}}u}}{{{{\left| {h'} \right|}^\alpha }}}} \right\|_{{L^q}\left( {{{\mathbb{H}}^n}} \right)}} + {2^\alpha }\mathop {\sup }\limits_{\left| {h'} \right| \ge \frac{1}{2}} {\left\| {\frac{{{\Delta _{Z,h'}}u}}{{{{\left| {h'} \right|}^\alpha }}}} \right\|_{{L^q}\left( {{{\mathbb{H}}^n}} \right)}}\\
   & \le {2^\alpha }\mathop {\sup }\limits_{0 < \left| {h'} \right| < \frac{1}{2}} {\left\| {\frac{{{\Delta _{Z,h'}}u}}{{{{\left| {h'} \right|}^\alpha }}}} \right\|_{{L^q}\left( {{{\mathbb{H}}^n}} \right)}} + {2^\alpha } \cdot 2\cdot{2^\alpha }{\left\| u \right\|_{{L^q}\left( {{{\mathbb{H}}^n}} \right)}}.
\end{align*}
By putting this estimate in \eqref{eq460} yields
\[\mathop {\sup }\limits_{0 < \left| h \right| < \frac{1}{2}} {\left\| {\frac{{{\Delta _{Z,h}}u}}{{{{\left| h \right|}^\alpha }}}} \right\|_{{L^q}\left( {{{\mathbb{H}}^n}} \right)}} \le \frac{1}{2}\mathop {\sup }\limits_{0 < \left| h \right| < \frac{1}{2}} {\left\| {\frac{{\Delta _{Z,h}^2u}}{{{{\left| h \right|}^\alpha }}}} \right\|_{{L^q}\left( {{{\mathbb{H}}^n}} \right)}} + {4^\alpha }{\left\| u \right\|_{{L^q}\left( {{{\mathbb{H}}^n}} \right)}} + {2^{\alpha  - 1}}\mathop {\sup }\limits_{0 < \left| {h'} \right| < \frac{1}{2}} {\left\| {\frac{{{\Delta _{Z,h'}}u}}{{{{\left| {h'} \right|}^\alpha }}}} \right\|_{{L^q}\left( {{{\mathbb{H}}^n}} \right)}}.\]
Recalling $\alpha<1$, the last term can be absorbed in the left-hand side and thus we get \eqref{eq461} with simple manipulations.

Finally, for every $h_0 >0$, it deduces
\[\mathop {\sup }\limits_{0 < \left| h \right|} {\left\| {\frac{{\Delta _{Z,h}^2u}}{{{{\left| h \right|}^\alpha }}}} \right\|_{{L^q}\left( {{{\mathbb{H}}^n}} \right)}} \le \mathop {\sup }\limits_{0 < \left| h \right| < {h_0}} {\left\| {\frac{{\Delta _{Z,h}^2u}}{{{{\left| h \right|}^\alpha }}}} \right\|_{{L^q}\left( {{{\mathbb{H}}^n}} \right)}} + 4h_0^{ - \alpha }{\left\| u \right\|_{{L^q}\left( {{{\mathbb{H}}^n}} \right)}},\]
so \eqref{eq462} can be obtained by combining \eqref{eq461} and the above inequality.
\end{proof}

%Le27来自于\cite{BL17}命题2.7的推广
\begin{lemma}\label{Le27}
Let $1 \le q< \infty$ and $0<\alpha< \beta \le1 $. If for some $h_0>0$, $u \in L^q\left( {{{\mathbb{H}}^n}} \right)$ satisfies
\[\mathop {\sup }\limits_{0 < \left| h \right| < {h_0}} {\left\| {\frac{{\Delta _{Z,h}u}}{{{{\left| h \right|}^\beta }}}} \right\|_{{L^q}\left( {{{\mathbb{H}}^n}} \right)}^q} < \infty, \]
then
\[\left[ u \right]_{H{W^{\alpha ,q}}\left( {{{\mathbb{H}}^n}} \right)}^q \le c\left( {\frac{{h_0^{\left( {\beta  - \alpha } \right)q}}}{{\beta  - \alpha }}\mathop {\sup }\limits_{0 < \left| h \right| < {h_0}} \left\| {\frac{{{\Delta _{Z,h}}u}}{{{{\left| h \right|}^\beta }}}} \right\|_{{L^q}\left( {{{\mathbb{H}}^n}} \right)}^q + \frac{{h_0^{ - \alpha q}}}{\alpha }\left\| u \right\|_{{L^q}\left( {{{\mathbb{H}}^n}} \right)}^q} \right),\]
where $c={c}(Q,q)>0$.
\end{lemma}

\begin{proof}
It follows by a direct computation that
\begin{align*}
   \left[ u \right]_{H{W^{\alpha ,q}}\left( {{{\mathbb{H}}^n}} \right)}^q& = \int_{\left\{ {\left| h \right| < {h_0}} \right\}} {\int_{{{\mathbb{H}}^n}} {\frac{{{{\left| {{\Delta _{Z,h}}u\left( \xi  \right)} \right|}^q}}}{{{{\left| h \right|}^{Q + \alpha q}}}}dhd\xi } }  + \int_{\left\{ {\left| h \right| \ge {h_0}} \right\}} {\int_{{{\mathbb{H}}^n}} {\frac{{{{\left| {{\Delta _{Z,h}}u\left( \xi  \right)} \right|}^q}}}{{{{\left| h \right|}^{Q + \alpha q}}}}dhd\xi } }  \\
   & \le \int_{\left\{ {\left| h \right| < {h_0}} \right\}} {\left( {\int_{{{\mathbb{H}}^n}} {\frac{{{{\left| {{\Delta _{Z,h}}u\left( \xi  \right)} \right|}^q}}}{{{{\left| h \right|}^{\beta q}}}}d\xi } } \right)\frac{{dh}}{{{{\left| h \right|}^{Q - \left( {\beta  - \alpha } \right)q}}}}}  + {2^q}\left\| u \right\|_{{L^q}\left( {{{\mathbb{H}}^n}} \right)}^q\int_{\left\{ {\left| h \right| \ge {h_0}} \right\}} {\frac{1}{{{{\left| h \right|}^{Q + \alpha q}}}}dh} \\
   & \le c\frac{{h_0^{\left( {\beta  - \alpha } \right)q}}}{{\beta  - \alpha }}\mathop {\sup }\limits_{0 < \left| h \right| < {h_0}} \left\| {\frac{{{\Delta _{Z,h}}u}}{{{{\left| h \right|}^\beta }}}} \right\|_{{L^q}\left( {{{\mathbb{H}}^n}} \right)}^q + \frac{{ch_0^{ - \alpha q}}}{\alpha }\left\| u \right\|_{{L^q}\left( {{{\mathbb{H}}^n}} \right)}^q.
\end{align*}
\end{proof}

%Le26来自于\cite{BLS18}引理2.6的推广
\begin{lemma}[Fractional Sobolev-type inequality]\label{Le26}
Let $1 \le q< \infty$ and $0<\alpha< \beta<1$. If for some $h_0>0$ and some Kor\'{a}nyi ball ${B_r} \subset {{\mathbb{H}}^n}$ with $r>h_0$, $u \in L_{loc}^q\left( {{{\mathbb{H}}^n}} \right)$ satisfies
\[\mathop {\sup }\limits_{0 < \left| h \right| < {h_0}} {\left\| {\frac{{\Delta _{Z,h}^2u}}{{{{\left| h \right|}^\beta }}}} \right\|_{{L^q}\left( {{B_r}} \right)}} < \infty, \]
then for every $\rho  > 0$ such that $\rho +h_0\le r$ we have
\begin{equation}\label{eq493}
\left[ u \right]_{H{W^{\alpha ,q}}\left( {{B_\rho }} \right)}^q \le {c_1}\left( {\mathop {\sup }\limits_{0 < \left| h \right| < {h_0}} \left\| {\frac{{\Delta _{Z,h}^2u}}{{{{\left| h \right|}^\beta }}}} \right\|_{{L^q}\left( {{B_r}} \right)}^q + \left\| u \right\|_{{L^q}\left( {{B_r}} \right)}^q} \right),
\end{equation}
and
\begin{equation}\label{eq494}
\mathop {\sup }\limits_{0 < \left| h \right| < {h_0}} \left\| {\frac{{{\Delta _{Z,h}}u}}{{{{\left| h \right|}^\beta }}}} \right\|_{{L^q}\left( {{B_\rho }} \right)}^q \le {c_2}\left( {\mathop {\sup }\limits_{0 < \left| h \right| < {h_0}} \left\| {\frac{{\Delta _{Z,h}^2u}}{{{{\left| h \right|}^\beta }}}} \right\|_{{L^q}\left( {{B_r}} \right)}^q + \left\| u \right\|_{{L^q}\left( {{B_r}} \right)}^q} \right),
\end{equation}
where ${{B_\rho }}$ is concentric with ${{B_r }}$. Here $c_1={c_1}(Q,q,\alpha,\beta,h_0)>0$ and $c_2={c_2}(Q,q,\beta,h_0)>0$ are constants that blow up as $\beta \to 1$, $\alpha \to \beta$ or $h_0 \to 0$.
\end{lemma}

\begin{proof}
We take a cut-off function $\psi  \in C_0^\infty \left( {{B_{\rho  + \frac{{{h_0}}}{3}}}} \right)$ satisfy
\[0 \le \psi  \le 1,\;\psi  = 1\;{\rm{in}}\;{B_\rho },\;\left| {{\nabla _H}\psi } \right| \le \frac{c}{{{h_0}}}\;{\rm{and}}\;\left| {{X_i}{X_j}\psi } \right| \le \frac{c}{{h_0^2}}.\]
Note that
\[\left| {{\Delta _{Z,h}}\psi } \right| \le \frac{c}{{{h_0}}}\left| h \right|\;{\rm{and}}\;\left| {\Delta _{Z,h}^2\psi } \right| \le \frac{c}{{h_0^2}}{\left| h \right|^2}\]
and
\begin{equation}\label{eq463}
\left[ u \right]_{H{W^{\alpha ,q}}\left( {{B_\rho }} \right)}^q \le \left[ {u\psi } \right]_{H{W^{\alpha ,q}}\left( {{{\mathbb{H}}^n}} \right)}^q.
\end{equation}
By applying Lemma \ref{Le27}, we get
\[\left[ u \psi\right]_{H{W^{\alpha ,q}}\left( {{{\mathbb{H}}^n}} \right)}^q \le c\left( {\frac{{h_0^{\left( {\beta  - \alpha } \right)q}}}{{\beta  - \alpha }}\mathop {\sup }\limits_{0 < \left| h \right| < {h_0}} \left\| {\frac{{{\Delta _{Z,h}}(u\psi)}}{{{{\left| h \right|}^\beta }}}} \right\|_{{L^q}\left( {{{\mathbb{H}}^n}} \right)}^q + \frac{{h_0^{ - \alpha q}}}{\alpha }\left\| u \psi\right\|_{{L^q}\left( {{{\mathbb{H}}^n}} \right)}^q} \right),\]
where $c=c(Q,q)>0$. We now use Lemma \ref{Le28} to obtain
\begin{equation}\label{eq464}
\left[ {u\psi } \right]_{H{W^{\alpha ,q}}\left( {{{\mathbb{H}}^n}} \right)}^q \le c\left( {\mathop {\sup }\limits_{0 < \left| h \right| < {h_0}} \left\| {\frac{{\Delta _{Z,h}^2\left( {u\psi } \right)}}{{{{\left| h \right|}^\beta }}}} \right\|_{{L^q}\left( {{{\mathbb{H}}^n}} \right)}^q + \left\| {u\psi } \right\|_{{L^q}\left( {{{\mathbb{H}}^n}} \right)}^q} \right),
\end{equation}
where $c=c(h_0, Q,q,\alpha,\beta)$. Consider
\[\Delta _{Z,h}^2\left( {u\psi } \right) = \psi \left( {\xi  {e^{2hZ}}} \right)\Delta _{Z,h}^2u + 2{\Delta _{Z,h}}u{\Delta _{Z,h}}\psi \left( {\xi   {e^{hZ}}} \right) + u\Delta _{Z,h}^2\psi ,\]
so for ${0 < \left| h \right| < {h_0}}$, we have
\begin{align}\label{eq465}
   &\left\| {\frac{{\Delta _{Z,h}^2\left( {u\psi } \right)}}{{{{\left| h \right|}^\beta }}}} \right\|_{{L^q}\left( {{{\mathbb{H}}^n}} \right)}^q\nonumber \\
   \le& c\left( {\left\| {\frac{{\psi \left( {\xi  {e^{2hZ}}} \right)\Delta _{Z,h}^2u}}{{{{\left| h \right|}^\beta }}}} \right\|_{{L^q}\left( {{{\mathbb{H}}^n}} \right)}^q + \left\| {\frac{{{\Delta _{Z,h}}u{\Delta _{Z,h}}\psi \left( {\xi   {e^{hZ}}} \right)}}{{{{\left| h \right|}^\beta }}}} \right\|_{{L^q}\left( {{{\mathbb{H}}^n}} \right)}^q + \left\| {\frac{{u\Delta _{Z,h}^2\psi }}{{{{\left| h \right|}^\beta }}}} \right\|_{{L^q}\left( {{{\mathbb{H}}^n}} \right)}^q} \right)\nonumber \\
   \le & c\left( {\left\| {\frac{{\Delta _{Z,h}^2u}}{{{{\left| h \right|}^\beta }}}} \right\|_{{L^q}\left( {{B_r}} \right)}^q + h_0^{\left( {1 - \beta } \right)q}\left\| {{\Delta _{Z,h}}u} \right\|_{{L^q}\left( {{B_{\rho  + \frac{{2{h_0}}}{3}}}} \right)}^q + h_0^{\left( {2 - \beta } \right)q}\left\| u \right\|_{{L^q}\left( {{B_r}} \right)}^q} \right)\nonumber \\
   \le& c\left( {\left\| {\frac{{\Delta _{Z,h}^2u}}{{{{\left| h \right|}^\beta }}}} \right\|_{{L^q}\left( {{B_r}} \right)}^q+\left\| u \right\|_{{L^q}\left( {{B_r}} \right)}^q} \right),
\end{align}
where $c=c(h_0, Q,q,\alpha,\beta)$. By combining \eqref{eq464} with \eqref{eq465}, we yield
\[\left[ u \right]_{H{W^{\alpha ,q}}\left( {{B_\rho }} \right)}^q \le c\left( {\mathop {\sup }\limits_{0 < \left| h \right| < {h_0}} \left\| {\frac{{\Delta _{Z,h}^2u}}{{{{\left| h \right|}^\beta }}}} \right\|_{{L^q}\left( {{B_r}} \right)}^q+\left\| u \right\|_{{L^q}\left( {{B_r}} \right)}^q} \right).\]

Next we prove the second inequality \eqref{eq494}. Consider
\[{\Delta _{Z,h}}\left( {u\psi } \right) = {\Delta _{Z,h}}u \cdot \psi  + u\left( {\xi   {e^{hZ}}} \right){\Delta _{Z,h}}\psi ,\]
so for ${0 < \left| h \right| < {h_0}}$, by using the above formula, Lemma \ref{Le28} and \eqref{eq465} we get
\begin{align*}
   \left\| {\frac{{{\Delta _{Z,h}}u}}{{{{\left| h \right|}^\beta }}}} \right\|_{{L^q}\left( {{B_\rho }} \right)}^q& \le \left\| {\frac{{{\Delta _{Z,h}}u \cdot \psi }}{{{{\left| h \right|}^\beta }}}} \right\|_{{L^q}\left( {{{\mathbb{H}}^n}} \right)}^q \\
   &  \le c\left( {\left\| {\frac{{{\Delta _{Z,h}}\left( {u\psi } \right)}}{{{{\left| h \right|}^\beta }}}} \right\|_{{L^q}\left( {{{\mathbb{H}}^n}} \right)}^q + \left\| {\frac{{u\left( {\xi  {e^{hZ}}} \right){\Delta _{Z,h}}\psi }}{{{{\left| h \right|}^\beta }}}} \right\|_{{L^q}\left( {{{\mathbb{H}}^n}} \right)}^q} \right)\\
   & \le c\left( {\left\| {\frac{{\Delta _{Z,h}^2\left( {u\psi } \right)}}{{{{\left| h \right|}^\beta }}}} \right\|_{{L^q}\left( {{{\mathbb{H}}^n}} \right)}^q + \left\| u \right\|_{{L^q}\left( {\rho  + {h_0}} \right)}^q+ \left\| {u\left( { \bullet {e^{hZ}}} \right)} \right\|_{{L^q}\left( {\rho  + \frac{{2{h_0}}}{3}} \right)}^q } \right)\\
  & \le c\left( {\left\| {\frac{{\Delta _{Z,h}^2\left( {u\psi } \right)}}{{{{\left| h \right|}^\beta }}}} \right\|_{{L^q}\left( {{{\mathbb{H}}^n}} \right)}^q + \left\| u \right\|_{{L^q}\left( {\rho  + {h_0}} \right)}^q} \right)\\
  & \le c\left( {\left\| {\frac{{\Delta _{Z,h}^2u}}{{{{\left| h \right|}^\beta }}}} \right\|_{{L^q}\left( {{B_r}} \right)}^q + \left\| u \right\|_{{L^q}\left( {{B_r}} \right)}^q} \right).
\end{align*}
\end{proof}

\begin{lemma}[\cite{DA04}, Theorem 1.1]\label{Le213}
Let $u \in L^p(\Omega)$ and $Z$ be a left-invariant vector field. If for $\alpha>1$,
\[\mathop {\sup }\limits_{0 < \left| h \right| < {h_0}} \int_{\Omega} {{{\left| {\frac{{\Delta _{Z,h}^2u}}{{{{\left| h \right|}^\alpha }}}} \right|}^p}d\xi }  \le c,\]
then
\[\int_{\Omega} {{{\left| {{\nabla _H}u} \right|}^p}d\xi }  \le c.\]
\end{lemma}

The following two lemmas are Lemma 2.7 and Theorem 1.1 of \cite{ZN26}, respectively.
\begin{lemma}[Caccioppoli-type inequality]\label{Le211}
Let $u$ be a weak subsolution to \eqref{eq0} in $B_2$, and $\omega  = {\left( {u - k} \right)_ + },\;k \in \mathbb{R}$. Then for any ${B_r} \equiv {B_r}\left( {{\xi _0}} \right) \subset {B_2} $ and non-negative function $\psi  \in C_0^\infty \left( {{B_r}} \right)$, it holds
\begin{align}\label{eq23}
   & \int_{{B_r}} {{\psi ^p}{{\left| {{\nabla _H}\omega } \right|}^p}d\xi }  + \int_{{B_r}} {\int_{{B_r}} {\frac{{{{\left| {\omega \left( \xi  \right)\psi \left( \xi  \right) - \omega \left( \eta  \right)\psi \left( \eta  \right)} \right|}^p}}}{{\left\| {{\eta ^{ - 1}} \circ \xi } \right\|_{{{\mathbb{H}}^n}}^{Q + sp}}} d\xi d\eta } } \nonumber \\
 \le &  c(\int_{{B_r}} {{\omega ^p}{{\left| {{\nabla _H}\psi } \right|}^p}d\xi }  + \int_{{B_r}} {\int_{{B_r}} {\max {{\{ \omega \left( \xi  \right),\omega \left( \eta  \right)\} }^p}\frac{{{{\left| {\psi \left( \xi  \right) - \psi \left( \eta  \right)} \right|}^p}}}{{\left\| {{\eta ^{ - 1}} \circ \xi } \right\|_{{{\mathbb{H}}^n}}^{Q + sp}}} d\xi d\eta } } \nonumber\\
 & + \mathop {{\rm{ess}}\;\sup }\limits_{\xi  \in {\rm{supp}}\;\psi } \int_{{{\mathbb{H}}^n}\backslash {B_r}} {\frac{{\omega {{\left( \eta  \right)}^{p - 1}}}}{{\left\| {{\eta ^{ - 1}} \circ \xi } \right\|_{{{\mathbb{H}}^n}}^{Q + sp}}}d\eta }  \cdot \int_{{B_r}} {\omega {\psi ^p}d\xi } ),
\end{align}
where $c = c\left( p \right)$. If $u$ is a weak supersolution to \eqref{eq0}, then \eqref{eq23} holds for $\omega  = {\left( {u - k} \right)_ - },\;k \in \mathbb{R}$.
\end{lemma}

\begin{lemma}[Local boundedness of weak subsolutions]\label{Le212}
Let $u \in HW_{loc}^{1,p}\left( \Omega  \right)$ $(1 < p < \infty )$ be a weak subsolution to \eqref{eq0}, and ${B_r} \equiv {B_r}\left( {{\xi _0}} \right) \subset \Omega ,\;r \in \left( {0,1} \right],$ here ${B_r}$ is a Kor\'{a}nyi ball defined by the C-C metric. Then for $\delta  \in \left( {0,1} \right]$, there exists a positive constant $c = c\left( {n,p,s} \right)$ such that
\begin{equation}\label{eq16}
\mathop {{\rm{ess}}\sup u}\limits_{{B_{\frac{r}{2}}}\left( {{\xi _0}} \right)}  \le \delta {{\rm{Tai}}{{\rm{l}}_{p - 1,sp,p}}}({u_ + };{\xi _0},\frac{r}{2}) + c{\delta ^{ - \frac{{\left( {p - 1} \right)\kappa }}{{p\left( {\kappa  - 1} \right)}}}}{\left( {\fint_{{B_r}} {u_ + ^pd\xi } } \right)^{\frac{1}{p}}},
\end{equation}
where
\[\kappa  = \left\{ \begin{array}{l}
\frac{Q}{{Q - p}},\;\;\;\;\;\;\;\;1 < p < Q,\\
\frac{{qQ}}{{p\left( {Q - q} \right)}},\;\;\;p \ge Q,\;\frac{{Qp}}{{Q + p}} < q < Q.
\end{array} \right.\;\;\;\]
\end{lemma}

\begin{lemma}[\cite{L19}]\label{LeA1}
Let $a,b \in {\mathbb{R}^N}$. Then for $p \ge 2$, we have
\begin{equation}\label{eqA1}
  \left\langle {{{\left| a \right|}^{p - 2}}a - {{\left| b \right|}^{p - 2}}b,a - b} \right\rangle  \ge {2^{2 - p}}{\left| {a - b} \right|^p};
\end{equation}
\begin{equation}\label{eqA2}
 \left\langle {{{\left| a \right|}^{p - 2}}a - {{\left| b \right|}^{p - 2}}b,a - b} \right\rangle  \ge \frac{4}{{{p^2}}}{\left| {{{\left| a \right|}^{\frac{{p - 2}}{2}}}a - {{\left| b \right|}^{\frac{{p - 2}}{2}}}b} \right|^2};
\end{equation}
\begin{equation}\label{eqA3}
 \left| {{{\left| a \right|}^{p - 2}}a - {{\left| b \right|}^{p - 2}}b} \right| \le \left( {p - 1} \right)\left( {{{\left| a \right|}^{\frac{{p - 2}}{2}}} + {{\left| b \right|}^{\frac{{p - 2}}{2}}}} \right)\left| {{{\left| a \right|}^{\frac{{p - 2}}{2}}}a - {{\left| b \right|}^{\frac{{p - 2}}{2}}}b} \right|.
\end{equation}
\end{lemma}

\begin{lemma}[\cite{BLS18}]\label{LeA2}
Let $ k,l,m,n \in \mathbb{R}$, $\alpha \ge 1$. Then for $p \ge 2$, we have
\begin{equation}\label{eqA4}
\left| {{J_p}\left( k \right) - {J_p}\left( l \right)} \right| \le \frac{{2\left( {p - 1} \right)}}{p}\left( {{{\left| k \right|}^{\frac{{p - 2}}{2}}} + {{\left| l \right|}^{\frac{{p - 2}}{2}}}} \right)\left| {{{\left| k \right|}^{\frac{{p - 2}}{2}}}k - {{\left| l \right|}^{\frac{{p - 2}}{2}}}l} \right|;
\end{equation}
\begin{align}\label{eqA5}
   &\left( {{J_p}\left( {k - l} \right) - {J_p}\left( {m - n} \right)} \right)\left( {{J_{\alpha  + 1}}\left( {k - m} \right) - {J_{\alpha  + 1}}\left( {l - n} \right)} \right) \nonumber\\
   \ge&  \frac{1}{{c\left( {p,\alpha } \right)}}{\left| {{{\left| {k - m} \right|}^{\frac{{\alpha  - 1}}{p}}}\left( {k - m} \right) - {{\left| {l - n} \right|}^{\frac{{\alpha  - 1}}{p}}}\left( {l - n} \right)} \right|^p};
\end{align}
\begin{align}\label{eqA6}
   & \left( {{J_p}\left( {k - l} \right) - {J_p}\left( {m - n} \right)} \right)\left( {{J_{\alpha  + 1}}\left( {k - m} \right) - {J_{\alpha  + 1}}\left( {l - n} \right)} \right) \nonumber\\
  \ge &  \frac{{2\left( {p - 1} \right)}}{{{p^2}}}{\left| {{{\left| {k - l} \right|}^{\frac{{p - 2}}{2}}}\left( {k - l} \right) - {{\left| {m - n} \right|}^{\frac{{p - 2}}{2}}}\left( {m - n} \right)} \right|^2}  \left( {{{\left| {k - m} \right|}^{\alpha  - 1}} + {{\left| {l - n} \right|}^{\alpha  - 1}}} \right).
\end{align}
\end{lemma}

\section{Iteration schme of Moser-type}\label{Section 3}
In this section, we employ the method of Nirenberg difference and Sobolev-type inequality on the Heisenberg group to establish the crucial Morrey-type iterative scheme for proving the theorems.

\begin{proposition}\label{Pro41}
Let $2\le p <\infty,\;0<s<1$ and $0\le \Lambda \le 1$. Assume $u \in HW_{{\rm{loc}}}^{1,p}\left( {{B_2}\left( {{\xi _0}} \right) \cap L_{sp}^{p - 1}\left( {{\mathbb{H}^n}} \right)} \right)$ is a weak solution of \eqref{eq0} in ${{B_2}\left( {{\xi _0}} \right)}$ satisfying
\begin{equation}\label{eq41}
{\left\| u \right\|_{{L^\infty }\left( {{B_1}\left( {{\xi _0}} \right)} \right)}} \le 1,\int_{{\mathbb{H}^n}\backslash {B_1}\left( {{\xi _0}} \right)} {\frac{{{{\left| {u\left( \eta  \right)} \right|}^{p - 1}}}}{{\left\| \eta  \right\|_{{\mathbb{H}^n}}^{Q + sp}}}d\eta }  \le 1.
\end{equation}
If $0<h_0<\frac{1}{10}$ and $R$ satisfy $4{h_0}<R\le 1-5{h_0}$, and ${\nabla _H}u \in {L^q}\left( {{B_{R + 4{h_0}}}\left( {{\xi _0}} \right)} \right)$ for some $q \ge p$, then
\begin{equation}\label{eq42}
 \mathop {\sup }\limits_{0 < \left| h \right| < {h_0}} \int_{{B_{R - {4h_0}}}} {{{\left| {\frac{{\Delta _{Z,h}^2u}}{{{{\left| h \right|}^{\frac{1}{{q + 1}} + 1}}}}} \right|}^{q + 1}}d\xi } \le c(1+\Lambda) \left(  {\int_{{B_{R + {4h_0}}}} {{{\left| {{\nabla _H}u} \right|}^q}d\xi } + 1} \right),
\end{equation}
where $c=c(Q,h_0,p,q,s)>0$.
\end{proposition}

\begin{proof}
Without loss of generality, assume $\xi_0=0 $. The proof is divided into five steps.

\textbf{Step1. Discrete differentiation of the equation.} Write
$$r = R - 4{h_0},\;d\mu  = d\mu \left( {\xi ,\eta } \right) = \frac{{d\xi d\eta }}{{\left\| {{\eta ^{ - 1}} \circ \xi } \right\|_{{{\mathbb{H}}^n}}^{Q + sp}}},$$
and take $\phi  \in HW_0^{1,p}\left( {{B_R}} \right)$ vanish outside ${B_{\frac{R+r}{2}}}$. Since $u$ is a weak solution of \eqref{eq0}, we have from Definition \ref{De26} that
\begin{equation}\label{eq43}
\int_{{B_R}} {{{\left| {{\nabla _H}u} \right|}^{p - 2}}{\nabla _H}u{\nabla _H}\phi d\xi }  +  \Lambda \int_{{{\mathbb{H}}^n}} {\int_{{{\mathbb{H}}^n}} {{\left({J_p}\left( {u\left( \xi  \right) - u\left( \eta  \right)} \right)\right)\left( {\phi \left( \xi  \right) - \phi \left( \eta  \right)} \right)}}d\mu }  = 0.
\end{equation}
Set $h \in \mathbb{R}\backslash \{0\}$ satisfying $\left| h \right|<h_0$. Taking $\phi  = \phi \left( {\xi {e^{ - hZ}}} \right)$ in \eqref{eq43} and using a change of variables, it yields
\begin{align}\label{eq44}
   & \int_{{B_R}} {{{\left| {{\nabla _H}u\left( {\xi {e^{hZ}}} \right)} \right|}^{p - 2}}{\nabla _H}u\left( {\xi {e^{hZ}}} \right){\nabla _H}\phi d\xi } \nonumber \\
   +&  \Lambda \int_{{{\mathbb{H}}^n}} {\int_{{{\mathbb{H}}^n}} {{{\left(J_p\left( {u\left( {\xi {e^{hZ}}} \right) - u\left( {\eta {e^{hZ}}} \right)} \right)\right)\left( {\phi \left( \xi  \right) - \phi \left( \eta  \right)} \right)}}d\mu } }  = 0.
\end{align}
Performing \eqref{eq44}-\eqref{eq43} and then dividing by $\left| h \right|$ gives
\begin{align}\label{eq45}
  & \int_{{B_R}} {\frac{{{{\left| {{\nabla _H}u\left( {\xi {e^{hZ}}} \right)} \right|}^{p - 2}}{\nabla _H}u\left( {\xi {e^{hZ}}} \right) - {{\left| {{\nabla _H}u} \right|}^{p - 2}}{\nabla _H}u}}{{\left| h \right|}}{\nabla _H}\phi d\xi } \nonumber \\
  + &  \Lambda \int_{{{\mathbb{H}}^n}} {\int_{{{\mathbb{H}}^n}} {\frac{{{J_p}\left( {u\left( {\xi {e^{hZ}}} \right) - u\left( {\eta {e^{hZ}}} \right)} \right) - {J_p}\left( {u\left( \xi  \right) - u\left( \eta  \right)} \right)}}{{\left| h \right|}}\left( {\phi \left( \xi  \right) - \phi \left( \eta  \right)} \right)d\mu } }  = 0
\end{align}
for any $\phi  \in HW_0^{1,p}\left( {{B_R}} \right)$ vanish outside ${B_{\frac{R+r}{2}}}$. Let $\varphi$ be a nonnegetive cut-off function between $B_r$ and ${B_{\frac{R+r}{2}}}$ satisfying
\[\varphi  \equiv 1\;\hbox{on}\;B_r,\;\;\varphi  \equiv 0\;\hbox{on}\;\mathbb{H}^n\backslash {B_{\frac{R+r}{2}}},\;\;\left| {{\nabla _H}\varphi} \right| \le \frac{c}{{R - r}} = \frac{c}{{4{h_0}}},\]
where $c=c(Q)>0$. Let $\alpha  \ge 1,\;\theta  > 0$. By taking the testing function
\[\phi \left( \xi  \right) = {J_{\alpha  + 1}}\left( {\frac{{u\left( {\xi {e^{hZ}}} \right) - u\left( \xi  \right)}}{{{{\left| h \right|}^\theta }}}} \right){\varphi ^p}\left( \xi  \right)\]
in \eqref{eq45}, we obtain
\begin{align}\label{eq46}
   0 =& \int_{{B_R}} {\frac{{{{\left| {{\nabla _H}u\left( {\xi {e^{hZ}}} \right)} \right|}^{p - 2}}{\nabla _H}u\left( {\xi {e^{hZ}}} \right) - {{\left| {{\nabla _H}u} \right|}^{p - 2}}{\nabla _H}u}}{{{{\left| h \right|}^{1 + \theta \alpha }}}}{\nabla _H}\left( {{J_{\alpha  + 1}}\left( {{{u\left( {\xi {e^{hZ}}} \right) - u}}} \right){\varphi ^p}} \right)d\xi } \nonumber \\
   & +  \Lambda \int_{{{\mathbb{H}}^n}} {\int_{{{\mathbb{H}}^n}} {\frac{{{J_p}\left( {u\left( {\xi {e^{hZ}}} \right) - u\left( {\eta {e^{hZ}}} \right)} \right) - {J_p}\left( {u\left( \xi  \right) - u\left( \eta  \right)} \right)}}{{{{\left| h \right|}^{1 + \theta \alpha }}}}} }\nonumber \\
    &\times \left( {{J_{\alpha  + 1}}\left( {{{u\left( {\xi {e^{hZ}}} \right) - u\left( \xi  \right)}}} \right){\varphi ^p}\left( \xi  \right) - {J_{\alpha  + 1}}\left( {{{u\left( {\eta {e^{hZ}}} \right) - u\left( \eta  \right)}}} \right){\varphi ^p}\left( \eta  \right)} \right)d\mu  \nonumber \\
    = :&I+ \Lambda J.
\end{align}

\textbf{Step2. Estimate of the local integral $I$.} Note that
\begin{align}\label{eq47}
   {I_1}: =& \left( {{{\left| {{\nabla _H}u\left( {\xi {e^{hZ}}} \right)} \right|}^{p - 2}}{\nabla _H}u\left( {\xi {e^{hZ}}} \right) - {{\left| {{\nabla _H}u} \right|}^{p - 2}}{\nabla _H}u} \right){\nabla _H}\left( {{J_{\alpha  + 1}}\left( {u\left( {\xi {e^{hZ}}} \right) - u} \right){\varphi ^p}} \right)\nonumber \\
   =& \left( {{{\left| {{\nabla _H}u\left( {\xi {e^{hZ}}} \right)} \right|}^{p - 2}}{\nabla _H}u\left( {\xi {e^{hZ}}} \right) - {{\left| {{\nabla _H}u} \right|}^{p - 2}}{\nabla _H}u} \right){\varphi ^p}{\nabla _H}\left( {{J_{\alpha  + 1}}\left( {u\left( {\xi {e^{hZ}}} \right) - u} \right)} \right) \nonumber \\
   & + \left( {{{\left| {{\nabla _H}u\left( {\xi {e^{hZ}}} \right)} \right|}^{p - 2}}{\nabla _H}u\left( {\xi {e^{hZ}}} \right) - {{\left| {{\nabla _H}u} \right|}^{p - 2}}{\nabla _H}u} \right){J_{\alpha  + 1}}\left( {u\left( {\xi {e^{hZ}}} \right) - u} \right){\nabla _H}\left( {{\varphi ^p}} \right)\nonumber \\
    \ge& \left( {{{\left| {{\nabla _H}u\left( {\xi {e^{hZ}}} \right)} \right|}^{p - 2}}{\nabla _H}u\left( {\xi {e^{hZ}}} \right) - {{\left| {{\nabla _H}u} \right|}^{p - 2}}{\nabla _H}u} \right){\varphi ^p}{\nabla _H}\left( {{J_{\alpha  + 1}}\left( {u\left( {\xi {e^{hZ}}} \right) - u} \right)} \right)\nonumber \\
   & - \left| {{{\left| {{\nabla _H}u\left( {\xi {e^{hZ}}} \right)} \right|}^{p - 2}}{\nabla _H}u\left( {\xi {e^{hZ}}} \right) - {{\left| {{\nabla _H}u} \right|}^{p - 2}}{\nabla _H}u} \right|{\left| {u\left( {\xi {e^{hZ}}} \right) - u} \right|^\alpha }\left| {{\nabla _H}\left( {{\varphi ^p}} \right)} \right|\nonumber \\
   : = &{I_{11}} - {I_{12}}.
\end{align}

\textbf{Estimate of ${I_{11}}$.} Since
\begin{align}\label{eq48}
   {\nabla _H}\left( {{{\left| u \right|}^{p - 2}}u} \right)=& {\left| u \right|^{p - 2}}{\nabla _H}u + u{\nabla _H}\left( {{{\left( {{u^2}} \right)}^{\frac{{p - 2}}{2}}}} \right) \nonumber \\
  = & {\left| u \right|^{p - 2}}{\nabla _H}u + \left( {p - 2} \right){\left| u \right|^{p - 2}}{\nabla _H}u\nonumber \\
   = &\left( {p - 1} \right){\left| u \right|^{p - 2}}{\nabla _H}u
\end{align}
for $p\ge 2$, it deduces
\begin{align*}\label{eq49}
   {I_{11}}& = \left( {{{\left| {{\nabla _H}u\left( {\xi {e^{hZ}}} \right)} \right|}^{p - 2}}{\nabla _H}u\left( {\xi {e^{hZ}}} \right) - {{\left| {{\nabla _H}u} \right|}^{p - 2}}{\nabla _H}u} \right){\varphi ^p}{\nabla _H}\left( {{J_{\alpha  + 1}}\left( {u\left( {\xi {e^{hZ}}} \right) - u} \right)} \right) \nonumber\\
   &  = \alpha \left( {{{\left| {{\nabla _H}u\left( {\xi {e^{hZ}}} \right)} \right|}^{p - 2}}{\nabla _H}u\left( {\xi {e^{hZ}}} \right) - {{\left| {{\nabla _H}u} \right|}^{p - 2}}{\nabla _H}u} \right){\nabla _H}\left( {u\left( {\xi {e^{hZ}}} \right) - u} \right){\left| {u\left( {\xi {e^{hZ}}} \right) - u} \right|^{\alpha  - 1}}{\varphi^p}.
\end{align*}
On the one hand, by virtue of \eqref{eqA2} and $\alpha \ge 1$, we have
\begin{equation}\label{eq410}
{I_{11}} \ge \frac{4}{{{p^2}}}{\left| {{{\left| {{\nabla _H}u\left( {\xi {e^{hZ}}} \right)} \right|}^{\frac{{p - 2}}{2}}}{\nabla _H}u\left( {\xi {e^{hZ}}} \right) - {{\left| {{\nabla _H}u} \right|}^{\frac{{p - 2}}{2}}}{\nabla _H}u} \right|^2}{\left| {u\left( {\xi {e^{hZ}}} \right) - u} \right|^{\alpha  - 1}}{\varphi ^p};
\end{equation}
on the other hand, we get by using \eqref{eqA1} and \eqref{eq48} that
\begin{align}\label{eq411}
   {I_{11}} \ge& \alpha {2^{2 - p}}{\left| {{\nabla _H}\left( {u\left( {\xi {e^{hZ}}} \right) - u} \right)} \right|^p}{\left| {u\left( {\xi {e^{hZ}}} \right) - u} \right|^{\alpha  - 1}}{\varphi^p} \nonumber\\
   =& \alpha {2^{2 - p}}{\left( {\frac{p}{{\alpha  + p - 1}}} \right)^p}{\left| {{\nabla _H}\left( {{{\left| {u\left( {\xi {e^{hZ}}} \right) - u} \right|}^{\frac{{\alpha  - 1}}{p}}}\left( {u\left( {\xi {e^{hZ}}} \right) - u} \right)} \right)} \right|^p}{\varphi ^p} \nonumber\\
\ge &\alpha {2^{2 - p}}{\left( {\frac{p}{{\alpha  + p - 1}}} \right)^p}\{ {2^{ - p}}{\left| {{\nabla _H}\left( {{{\left| {u\left( {\xi {e^{hZ}}} \right) - u} \right|}^{\frac{{\alpha  - 1}}{p}}}\left( {u\left( {\xi {e^{hZ}}} \right) - u} \right)\varphi } \right)} \right|^p}\nonumber\\
& - {\left| {{{\left| {u\left( {\xi {e^{hZ}}} \right) - u} \right|}^{\frac{{\alpha  - 1}}{p}}}\left( {u\left( {\xi {e^{hZ}}} \right) - u} \right)} \right|^p}{\left| {{\nabla _H}\varphi } \right|^p}\} .
\end{align}

\textbf{Estimate of ${I_{12}}$.} By means of $p \ge 2$, \eqref{eqA3}, Young's inequality and \eqref{eq410}, we get
\begin{align}\label{eq412}
   {I_{12}} =& \left| {{{\left| {{\nabla _H}u\left( {\xi {e^{hZ}}} \right)} \right|}^{p - 2}}{\nabla _H}u\left( {\xi {e^{hZ}}} \right) - {{\left| {{\nabla _H}u} \right|}^{p - 2}}{\nabla _H}u} \right|{\left| {u\left( {\xi {e^{hZ}}} \right) - u} \right|^\alpha }\left| {{\nabla _H}\left( {{\varphi ^p}} \right)} \right|\nonumber\\
 \le  &  \left( {p - 1} \right)\left( {{{\left| {{\nabla _H}u\left( {\xi {e^{hZ}}} \right)} \right|}^{\frac{{p - 2}}{2}}} + {{\left| {{\nabla _H}u} \right|}^{\frac{{p - 2}}{2}}}} \right)\nonumber\\
  & \times \left| {{{\left| {{\nabla _H}u\left( {\xi {e^{hZ}}} \right)} \right|}^{\frac{{p - 2}}{2}}}{\nabla _H}u\left( {\xi {e^{hZ}}} \right) - {{\left| {{\nabla _H}u} \right|}^{\frac{{p - 2}}{2}}}{\nabla _H}u} \right|{\left| {u\left( {\xi {e^{hZ}}} \right) - u} \right|^\alpha }2{\varphi ^{\frac{p}{2}}}\left| {{\nabla _H}\left( {{\varphi ^{\frac{p}{2}}}} \right)} \right| \nonumber\\
   = &\left( {\left( {p - 1} \right)\left( {{{\left| {{\nabla _H}u\left( {\xi {e^{hZ}}} \right)} \right|}^{\frac{{p - 2}}{2}}} + {{\left| {{\nabla _H}u} \right|}^{\frac{{p - 2}}{2}}}} \right){{\left| {u\left( {\xi {e^{hZ}}} \right) - u} \right|}^{\frac{{\alpha  + 1}}{2}}}\left| {{\nabla _H}\left( {{\varphi ^{\frac{p}{2}}}} \right)} \right|} \right) \nonumber\\
   & \times \left( {\left| {{{\left| {{\nabla _H}u\left( {\xi {e^{hZ}}} \right)} \right|}^{\frac{{p - 2}}{2}}}{\nabla _H}u\left( {\xi {e^{hZ}}} \right) - {{\left| {{\nabla _H}u} \right|}^{\frac{{p - 2}}{2}}}{\nabla _H}u} \right|{{\left| {u\left( {\xi {e^{hZ}}} \right) - u} \right|}^{\frac{{\alpha  - 1}}{2}}}2{\varphi ^{\frac{p}{2}}}} \right)\nonumber\\
   \le& c\left( {p,\varepsilon } \right){\left( {{{\left| {{\nabla _H}u\left( {\xi {e^{hZ}}} \right)} \right|}^{\frac{{p - 2}}{2}}} + {{\left| {{\nabla _H}u} \right|}^{\frac{{p - 2}}{2}}}} \right)^2}{\left| {u\left( {\xi {e^{hZ}}} \right) - u} \right|^{\alpha  + 1}}{\left| {{\nabla _H}\left( {{\varphi ^{\frac{p}{2}}}} \right)} \right|^2} \nonumber\\
   & + \varepsilon {\left| {{{\left| {{\nabla _H}u\left( {\xi {e^{hZ}}} \right)} \right|}^{\frac{{p - 2}}{2}}}{\nabla _H}u\left( {\xi {e^{hZ}}} \right) - {{\left| {{\nabla _H}u} \right|}^{\frac{{p - 2}}{2}}}{\nabla _H}u} \right|^2}{\left| {u\left( {\xi {e^{hZ}}} \right) - u} \right|^{\alpha  - 1}}{\varphi ^p}\nonumber\\
    \le& c\left( {p,\varepsilon } \right){\left( {{{\left| {{\nabla _H}u\left( {\xi {e^{hZ}}} \right)} \right|}^{\frac{{p - 2}}{2}}} + {{\left| {{\nabla _H}u} \right|}^{\frac{{p - 2}}{2}}}} \right)^2}{\left| {u\left( {\xi {e^{hZ}}} \right) - u} \right|^{\alpha  + 1}}{\left| {{\nabla _H}\left( {{\varphi ^{\frac{p}{2}}}} \right)} \right|^2} + \frac{{\varepsilon {p^2}}}{4}{I_{11}}.
\end{align}
Then, by choosing $\varepsilon  \in \left( {0,\frac{4}{{{p^2}}}} \right)$ small enough and combining \eqref{eq47} with \eqref{eq412}, we have from \eqref{eq411} that
\begin{align}\label{eq413}
   {I_1} \ge& c\left( p \right){I_{11}} - c\left( p \right){\left( {{{\left| {{\nabla _H}u\left( {\xi {e^{hZ}}} \right)} \right|}^{\frac{{p - 2}}{2}}} + {{\left| {{\nabla _H}u} \right|}^{\frac{{p - 2}}{2}}}} \right)^2}{\left| {u\left( {\xi {e^{hZ}}} \right) - u} \right|^{\alpha  + 1}}{\left| {{\nabla _H}\left( {{\varphi ^{\frac{p}{2}}}} \right)} \right|^2} \nonumber\\
   \ge&  c\left( p \right)\alpha {2^{2 - p}}{\left( {\frac{p}{{\alpha  + p - 1}}} \right)^p}\{ {2^{ - p}}{\left| {{\nabla _H}\left( {{{\left| {u\left( {\xi {e^{hZ}}} \right) - u} \right|}^{\frac{{\alpha  - 1}}{p}}}\left( {u\left( {\xi {e^{hZ}}} \right) - u} \right)\varphi } \right)} \right|^p}\nonumber\\
   & - {\left| {{{\left| {u\left( {\xi {e^{hZ}}} \right) - u} \right|}^{\frac{{\alpha  - 1}}{p}}}\left( {u\left( {\xi {e^{hZ}}} \right) - u} \right)} \right|^p}{\left| {{\nabla _H}\varphi } \right|^p}\} \nonumber\\
  & - c\left( p \right){\left( {{{\left| {{\nabla _H}u\left( {\xi {e^{hZ}}} \right)} \right|}^{\frac{{p - 2}}{2}}} + {{\left| {{\nabla _H}u} \right|}^{\frac{{p - 2}}{2}}}} \right)^2}{\left| {u\left( {\xi {e^{hZ}}} \right) - u} \right|^{\alpha  + 1}}{\left| {{\nabla _H}\left( {{\varphi ^{\frac{p}{2}}}} \right)} \right|^2}.
\end{align}
Therefore, it yields
\begin{align}\label{eq414}
I =   & \int_{{B_R}} {\frac{{{I_1}}}{{{{\left| h \right|}^{1 + \theta \alpha }}}}d\xi }  \nonumber\\
  \ge&  c\left( {p,\alpha } \right)\int_{{B_R}} {{{\left| {{\nabla _H}\left( {\frac{{{{\left| {u\left( {\xi {e^{hZ}}} \right) - u} \right|}^{\frac{{\alpha  - 1}}{p}}}\left( {u\left( {\xi {e^{hZ}}} \right) - u} \right)\varphi }}{{{{\left| h \right|}^{\frac{{1 + \theta \alpha }}{p}}}}}} \right)} \right|}^p}d\xi }  \nonumber\\
  & - c\left( {p,\alpha } \right)\int_{{B_R}} {\frac{{{{\left| {{{\left| {u\left( {\xi {e^{hZ}}} \right) - u} \right|}^{\frac{{\alpha  - 1}}{p}}}\left( {u\left( {\xi {e^{hZ}}} \right) - u} \right)} \right|}^p}{{\left| {{\nabla _H}\varphi } \right|}^p}}}{{{{\left| h \right|}^{1 + \theta \alpha }}}}d\xi } \nonumber\\
 & - c\left( p \right)\int_{{B_R}} {\frac{{{{\left( {{{\left| {{\nabla _H}u\left( {\xi {e^{hZ}}} \right)} \right|}^{\frac{{p - 2}}{2}}} + {{\left| {{\nabla _H}u} \right|}^{\frac{{p - 2}}{2}}}} \right)}^2}{{\left| {u\left( {\xi {e^{hZ}}} \right) - u} \right|}^{\alpha  + 1}}{{\left| {{\nabla _H}\left( {{\varphi ^{\frac{p}{2}}}} \right)} \right|}^2}}}{{{{\left| h \right|}^{1 + \theta \alpha }}}}d\xi }  \nonumber\\
  : =& c\left( {p,\alpha } \right)\int_{{B_R}} {{{\left| {{\nabla _H}\left( {\frac{{{{\left| {u\left( {\xi {e^{hZ}}} \right) - u} \right|}^{\frac{{\alpha  - 1}}{p}}}\left( {u\left( {\xi {e^{hZ}}} \right) - u} \right)\varphi }}{{{{\left| h \right|}^{\frac{{1 + \theta \alpha }}{p}}}}}} \right)} \right|}^p}d\xi }  \nonumber\\
  &- c\left( {p,\alpha } \right){I_{13}} - c\left( p \right){I_{14}}.
\end{align}

\textbf{Estimate of ${I_{13}}$.} Note that
\begin{align}\label{eq4150}
   {I_{13}} = & \int_{{B_R}} {\frac{{{{\left| {{{\left| {u\left( {\xi {e^{hZ}}} \right) - u} \right|}^{\frac{{\alpha  - 1}}{p}}}\left( {u\left( {\xi {e^{hZ}}} \right) - u} \right)} \right|}^p}{{\left| {{\nabla _H}\varphi } \right|}^p}}}{{{{\left| h \right|}^{1 + \theta \alpha }}}}d\xi } \nonumber \\
   =&  \int_{{B_R}} {\frac{{{{\left| {{\Delta _{Z,h}}u} \right|}^{\alpha  + p - 1}}{{\left| {{\nabla _H}\varphi} \right|}^p}}}{{{{\left| h \right|}^{1 + \theta \alpha }}}}d\xi }\nonumber \\
    \le& {\left( {\frac{c}{{4{h_0}}}} \right)^p}\int_{{B_R}} {\frac{{{{\left| {{\Delta _{Z,h}}u} \right|}^{\alpha  + p - 1}}}}{{{{\left| h \right|}^{1 + \theta \alpha }}}}d\xi } ,
\end{align}
where we have used the properties of $\varphi$. When $p=2$, it follows from \eqref{eq4150} and ${\left\| u \right\|_{{L^\infty }\left( {{B_1}} \right)}} \le 1$ that
\begin{align}\label{eq4151}
   {I_{13}}& \le {\left( {\frac{c}{{4{h_0}}}} \right)^2}\int_{{B_R}} {\frac{{{{\left| {{\Delta _{Z,h}}u} \right|}^{\alpha  + 1}}}}{{{{\left| h \right|}^{1 + \theta \alpha }}}}d\xi } \nonumber \\
  \le &  {\left( {\frac{c}{{4{h_0}}}} \right)^2}\int_{{B_R}} {\frac{{2{{\left\| u \right\|}_{{L^\infty }\left( {{B_1}} \right)}}{{\left| {{\Delta _{Z,h}}u} \right|}^\alpha }}}{{{{\left| h \right|}^{1 + \theta \alpha }}}}d\xi } \nonumber \\
    \le& 2{\left( {\frac{c}{{4{h_0}}}} \right)^2}\int_{{B_R}} {\frac{{{{\left| {{\Delta _{Z,h}}u} \right|}^\alpha }}}{{{{\left| h \right|}^{1 + \theta \alpha }}}}d\xi }.
\end{align}
When $p>2$, by applying \eqref{eq4150}, Young's inequality with exponents $\frac{q}{p-2}$ and $\frac{q}{q-p+2}$ and ${\left\| u \right\|_{{L^\infty }\left( {{B_1}} \right)}} \le 1$, we get
\begin{align}\label{eq4152}
   {I_{13}}& \le {\left( {\frac{c}{{4{h_0}}}} \right)^p}\int_{{B_R}} {\frac{{{{\left| {{\Delta _{Z,h}}u} \right|}^{\alpha  + p - 1}}}}{{{{\left| h \right|}^{1 + \theta \alpha }}}}d\xi }  \nonumber \\
   &\le {\left( {\frac{c}{{4{h_0}}}} \right)^p}\left( {\int_{{B_R}} {\frac{{{{\left| {{\Delta _{Z,h}}u} \right|}^{\frac{{\alpha q}}{{q - p + 2}}}}}}{{{{\left| h \right|}^{\frac{{\left( {1 + \theta \alpha } \right)q}}{{q - p + 2}}}}}}d\xi }  + \int_{{B_R}} {{{\left| {{\Delta _{Z,h}}u} \right|}^{\frac{{\left( {p - 1} \right)q}}{{p - 2}}}}d\xi } } \right)\nonumber \\
    &\le c\left( {Q,{h_0},p,q} \right)\left( {\int_{{B_R}} {\frac{{{{\left| {{\Delta _{Z,h}}u} \right|}^{\frac{{\alpha q}}{{q - p + 2}}}}}}{{{{\left| h \right|}^{\frac{{\left( {1 + \theta \alpha } \right)q}}{{q - p + 2}}}}}}d\xi }  + 1} \right).
\end{align}
Therefore, it deduces by combining \eqref{eq4151} and \eqref{eq4152} that
\begin{equation}\label{eq415}
 {I_{13}} \le c\left( {Q,{h_0},p,q} \right)\left( {\int_{{B_R}} {\frac{{{{\left| {{\Delta _{Z,h}}u} \right|}^{\frac{{\alpha q}}{{q - p + 2}}}}}}{{{{\left| h \right|}^{\frac{{\left( {1 + \theta \alpha } \right)q}}{{q - p + 2}}}}}}d\xi }  + 1} \right).
\end{equation}

\textbf{Estimate of ${I_{14}}$.} By straightforward calculations, we obtain
\begin{align}\label{eq416}
   {I_{14}}&  = \int_{{B_R}} {\frac{{{{\left( {{{\left| {{\nabla _H}u\left( {\xi {e^{hZ}}} \right)} \right|}^{\frac{{p - 2}}{2}}} + {{\left| {{\nabla _H}u} \right|}^{\frac{{p - 2}}{2}}}} \right)}^2}{{\left| {u\left( {\xi {e^{hZ}}} \right) - u} \right|}^{\alpha  + 1}}{{\left| {{\nabla _H}\left( {{\varphi ^{\frac{p}{2}}}} \right)} \right|}^2}}}{{{{\left| h \right|}^{1 + \theta \alpha }}}}d\xi } \nonumber \\
   &  \le 4\cdot \frac{p^2}{4}\int_{{B_R}} {\frac{{\left( {{{\left| {{\nabla _H}u\left( {\xi {e^{hZ}}} \right)} \right|}^{p - 2}} + {{\left| {{\nabla _H}u} \right|}^{p - 2}}} \right){{\left| {{\Delta _{Z,h}}u} \right|}^{\alpha  + 1}}{{\left| {{\nabla _H}\varphi } \right|}^2}}}{{{{\left| h \right|}^{1 + \theta \alpha }}}}d\xi }\nonumber \\
    &= {p^2}\int_{{B_R}} {\frac{{{{\left| {{\nabla _H}u\left( {\xi {e^{hZ}}} \right)} \right|}^{p - 2}}{{\left| {{\Delta _{Z,h}}u} \right|}^{\alpha  + 1}}{{\left| {{\nabla _H}\varphi } \right|}^2}}}{{{{\left| h \right|}^{1 + \theta \alpha }}}}d\xi }  + {p^2}\int_{{B_R}} {\frac{{{{\left| {{\nabla _H}u} \right|}^{p - 2}}{{\left| {{\Delta _{Z,h}}u} \right|}^{\alpha  + 1}}{{\left| {{\nabla _H}\varphi } \right|}^2}}}{{{{\left| h \right|}^{1 + \theta \alpha }}}}d\xi }\nonumber \\
    &: = {p^2}\left( {{I_{15}} + {I_{16}}} \right).
\end{align}

\textbf{Estimate of ${I_{15}}$ and ${I_{16}}$.} When $p=2$, it gets that from ${\left\| u \right\|_{{L^\infty }\left( {{B_1}} \right)}} \le 1$ and the properties of $\varphi$
\begin{align}\label{eq417}
  {I_{15}}  & = \int_{{B_R}} {\frac{{{{\left| {{\Delta _{Z,h}}u} \right|}^{\alpha  + 1}}{{\left| {{\nabla _H}\varphi } \right|}^2}}}{{{{\left| h \right|}^{1 + \theta \alpha }}}}d\xi }  \le 2{\left\| u \right\|_{{L^\infty }\left( {{B_1}} \right)}}\int_{{B_R}} {\frac{{{{\left| {{\Delta _{Z,h}}u} \right|}^\alpha }{{\left| {{\nabla _H}\varphi } \right|}^2}}}{{{{\left| h \right|}^{1 + \theta \alpha }}}}d\xi }  \nonumber\\
   &  \le 2{\left( {\frac{c}{{4{h_0}}}} \right)^2}\int_{{B_R}} {\frac{{{{\left| {{\Delta _{Z,h}}u} \right|}^\alpha }}}{{{{\left| h \right|}^{1 + \theta \alpha }}}}d\xi } .
\end{align}
When $p>2$, by applying Young's inequality with exponents $\frac{q}{p-2}$ and $\frac{q}{q-p+2}$ and ${\left\| u \right\|_{{L^\infty }\left( {{B_1}} \right)}} \le 1$, we have
\begin{align}\label{eq418}
   {I_{15}}& = \int_{{B_R}} {\frac{{{{\left| {{\nabla _H}u\left( {\xi {e^{hZ}}} \right)} \right|}^{p - 2}}{{\left| {{\Delta _{Z,h}}u} \right|}^{\alpha  + 1}}{{\left| {{\nabla _H}\varphi } \right|}^2}}}{{{{\left| h \right|}^{1 + \theta \alpha }}}}d\xi }  \nonumber\\
   &  \le c(p,q)\int_{{B_R}} {{{\left| {{\nabla _H}u\left( {\xi {e^{hZ}}} \right)} \right|}^q}d\xi }  + c(p,q)\int_{{B_R}} {{{\left( {\frac{{{{\left| {{\Delta _{Z,h}}u} \right|}^{\alpha  + 1}}{{\left| {{\nabla _H}\varphi } \right|}^2}}}{{{{\left| h \right|}^{1 + \theta \alpha }}}}} \right)}^{\frac{q}{{q - p + 2}}}}d\xi }\nonumber\\
   & \le c(p,q)\int_{{B_R}} {{{\left| {{\nabla _H}u\left( {\xi {e^{hZ}}} \right)} \right|}^q}d\xi }  + c(p,q){\left( {\frac{C}{{4{h_0}}}} \right)^{\frac{{2q}}{{q - p + 2}}}}{\left( {2{{\left\| u \right\|}_{{L^\infty }\left( {{B_1}} \right)}}} \right)^{\frac{q}{{q - p + 2}}}}\int_{{B_R}} {\frac{{{{\left| {{\Delta _{Z,h}}u} \right|}^{\frac{{\alpha q}}{{q - p + 2}}}}}}{{{{\left| h \right|}^{\frac{{\left( {1 + \theta \alpha } \right)q}}{{q - p + 2}}}}}}d\xi } \nonumber\\
   & \le c(p,q)\int_{{B_R}} {{{\left| {{\nabla _H}u\left( {\xi {e^{hZ}}} \right)} \right|}^q}d\xi }  + c\left( {Q,{h_0},p,q} \right)\int_{{B_R}} {\frac{{{{\left| {{\Delta _{Z,h}}u} \right|}^{\frac{{\alpha q}}{{q - p + 2}}}}}}{{{{\left| h \right|}^{\frac{{\left( {1 + \theta \alpha } \right)q}}{{q - p + 2}}}}}}d\xi } .
\end{align}
Therefore, by combining \eqref{eq417} and \eqref{eq418}, it yields that for any $p \ge2$
\begin{equation}\label{eq419}
 {I_{15}} \le c(p,q)\int_{{B_R}} {{{\left| {{\nabla _H}u\left( {\xi {e^{hZ}}} \right)} \right|}^q}d\xi }  + c\left( {Q,{h_0},p,q} \right)\int_{{B_R}} {\frac{{{{\left| {{\Delta _{Z,h}}u} \right|}^{\frac{{\alpha q}}{{q - p + 2}}}}}}{{{{\left| h \right|}^{\frac{{\left( {1 + \theta \alpha } \right)q}}{{q - p + 2}}}}}}d\xi } .
\end{equation}
Similarly, it follows
\begin{equation}\label{eq420}
{I_{16}} \le c(p,q)\int_{{B_R}} {{{\left| {{\nabla _H}u} \right|}^q}d\xi }  + c\left( {Q,{h_0},p,q} \right)\int_{{B_R}} {\frac{{{{\left| {{\Delta _{Z,h}}u} \right|}^{\frac{{\alpha q}}{{q - p + 2}}}}}}{{{{\left| h \right|}^{\frac{{\left( {1 + \theta \alpha } \right)q}}{{q - p + 2}}}}}}d\xi } .
\end{equation}
Substituting \eqref{eq419} and \eqref{eq420} into \eqref{eq416} gives
\begin{equation}\label{eq421}
 {I_{14}} \le c(p,q)\int_{{B_R}} {{{\left| {{\nabla _H}u\left( {\xi {e^{hZ}}} \right)} \right|}^q}d\xi }  + c(p,q)\int_{{B_R}} {{{\left| {{\nabla _H}u} \right|}^q}d\xi }  + c\left( {Q,{h_0},p,q} \right)\int_{{B_R}} {\frac{{{{\left| {{\Delta _{Z,h}}u} \right|}^{\frac{{\alpha q}}{{q - p + 2}}}}}}{{{{\left| h \right|}^{\frac{{\left( {1 + \theta \alpha } \right)q}}{{q - p + 2}}}}}}d\xi } .
\end{equation}
Thus, Plugging \eqref{eq415} and \eqref{eq421} into \eqref{eq414} yields
\begin{align}\label{eq422}
   I \ge& c(p,\alpha)\int_{{B_R}} {{{\left| {{\nabla _H}\left( {\frac{{{{\left| {u\left( {\xi {e^{hZ}}} \right) - u} \right|}^{\frac{{\alpha  - 1}}{p}}}\left( {u\left( {\xi {e^{hZ}}} \right) - u} \right)\varphi }}{{{{\left| h \right|}^{\frac{{1 + \theta \alpha }}{p}}}}}} \right)} \right|}^p}d\xi }  - c(p,q,\alpha)\int_{{B_R}} {{{\left| {{\nabla _H}u\left( {\xi {e^{hZ}}} \right)} \right|}^q}d\xi }  \nonumber\\
   &  - c(p,q,\alpha)\int_{{B_R}} {{{\left| {{\nabla _H}u} \right|}^q}d\xi }  - c\left( {Q,{h_0},p,q} \right)\int_{{B_R}} {\frac{{{{\left| {{\Delta _{Z,h}}u} \right|}^{\frac{{\alpha q}}{{q - p + 2}}}}}}{{{{\left| h \right|}^{\frac{{\left( {1 + \theta \alpha } \right)q}}{{q - p + 2}}}}}}d\xi }  - c\left( {Q,{h_0},p,q} \right).
\end{align}

\textbf{Step3. Estimate of the nonlocal integral $J$.} By spliting ${{\mathbb{H}}^n} \times {{\mathbb{H}}^n}$ into
\[\left( {{B_R} \times {B_R}} \right) \cup \left( {{{\mathbb{H}}^n}\backslash {B_R} \times {B_R}} \right) \cup \left( {{B_R} \times {{\mathbb{H}}^n}\backslash {B_R}} \right) \cup \left( {{{\mathbb{H}}^n}\backslash {B_R} \times {{\mathbb{H}}^n}\backslash {B_R}} \right),\]
and using $\varphi \equiv 0$ on ${{\mathbb{H}}^n}\backslash {B_{\frac{R+r}{2}}}$, we obtain
\begin{align}\label{eq423}
  J = & \int_{{{\mathbb{H}}^n}} {\int_{{{\mathbb{H}}^n}} {\frac{{{J_p}\left( {u\left( {\xi {e^{hZ}}} \right) - u\left( {\eta {e^{hZ}}} \right)} \right) - {J_p}\left( {u\left( \xi  \right) - u\left( \eta  \right)} \right)}}{{{{\left| h \right|}^{1 + \theta \alpha }}}}} } \nonumber\\
   &  \times \left( {{J_{\alpha  + 1}}\left( {u\left( {\xi {e^{hZ}}} \right) - u\left( \xi  \right)} \right){\varphi ^p}\left( \xi  \right) - {J_{\alpha  + 1}}\left( {u\left( {\eta {e^{hZ}}} \right) - u\left( \eta  \right)} \right){\varphi ^p}\left( \eta  \right)} \right)d\mu \nonumber\\
    = &\int_{{B_R}} {\int_{{B_R}} {\frac{{{J_p}\left( {u\left( {\xi {e^{hZ}}} \right) - u\left( {\eta {e^{hZ}}} \right)} \right) - {J_p}\left( {u\left( \xi  \right) - u\left( \eta  \right)} \right)}}{{{{\left| h \right|}^{1 + \theta \alpha }}}}} }\nonumber\\
   &\times \left( {{J_{\alpha  + 1}}\left( {{\Delta _{Z,h}}u\left( \xi  \right)} \right){\varphi ^p}\left( \xi  \right) - {J_{\alpha  + 1}}\left( {{\Delta _{Z,h}}\left( \eta  \right)} \right){\varphi ^p}\left( \eta  \right)} \right)d\mu\nonumber\\
   & + \int_{{{\mathbb{H}}^n}\backslash {B_R}} {\int_{{B_{\frac{{R + r}}{2}}}} {\frac{{{J_p}\left( {u\left( {\xi {e^{hZ}}} \right) - u\left( {\eta {e^{hZ}}} \right)} \right) - {J_p}\left( {u\left( \xi  \right) - u\left( \eta  \right)} \right)}}{{{{\left| h \right|}^{1 + \theta \alpha }}}}{J_{\alpha  + 1}}\left( {{\Delta _{Z,h}}u\left( \xi  \right)} \right){\varphi ^p}\left( \xi  \right)d\mu } }\nonumber\\
   & - \int_{{B_{\frac{{R + r}}{2}}}} {\int_{{{\mathbb{H}}^n}\backslash {B_R}} {\frac{{{J_p}\left( {u\left( {\xi {e^{hZ}}} \right) - u\left( {\eta {e^{hZ}}} \right)} \right) - {J_p}\left( {u\left( \xi  \right) - u\left( \eta  \right)} \right)}}{{{{\left| h \right|}^{1 + \theta \alpha }}}}{J_{\alpha  + 1}}\left( {{\Delta _{Z,h}}\left( \eta  \right)} \right){\varphi ^p}\left( \eta  \right)d\mu } } \nonumber\\
   : = &{J_1} + {J_2} - {J_3}.
\end{align}

\textbf{Estimate of ${J_{1}}$.} We start by observing that
\begin{align*}
   & {J_{\alpha  + 1}}\left( {{\Delta _{Z,h}}u\left( \xi  \right)} \right){\varphi ^p}\left( \xi  \right) - {J_{\alpha  + 1}}\left( {{\Delta _{Z,h}}\left( \eta  \right)} \right){\varphi ^p}\left( \eta  \right) \\
   =& \frac{{{J_{\alpha  + 1}}\left( {{\Delta _{Z,h}}u\left( \xi  \right)} \right) - {J_{\alpha  + 1}}\left( {{\Delta _{Z,h}}\left( \eta  \right)} \right)}}{2}\left( {{\varphi ^p}\left( \xi  \right) + {\varphi ^p}\left( \eta  \right)} \right)\\
   & + \frac{{{J_{\alpha  + 1}}\left( {{\Delta _{Z,h}}u\left( \xi  \right)} \right) + {J_{\alpha  + 1}}\left( {{\Delta _{Z,h}}\left( \eta  \right)} \right)}}{2}\left( {{\varphi ^p}\left( \xi  \right) - {\varphi ^p}\left( \eta  \right)} \right).
\end{align*}
Then
\begin{align}\label{eq450}
J'_1:=&\left( {{J_p}\left( {u\left( {\xi {e^{hZ}}} \right) - u\left( {\eta {e^{hZ}}} \right)} \right) - {J_p}\left( {u\left( \xi  \right) - u\left( \eta  \right)} \right)} \right)\left( {{J_{\alpha  + 1}}\left( {{\Delta _{Z,h}}u\left( \xi  \right)} \right){\varphi ^p}\left( \xi  \right) - {J_{\alpha  + 1}}\left( {{\Delta _{Z,h}}\left( \eta  \right)} \right){\varphi ^p}\left( \eta  \right)} \right)\nonumber\\
  \ge &\left( {{J_p}\left( {u\left( {\xi {e^{hZ}}} \right) - u\left( {\eta {e^{hZ}}} \right)} \right) - {J_p}\left( {u\left( \xi  \right) - u\left( \eta  \right)} \right)} \right)\nonumber\\
  &\times\left( {{J_{\alpha  + 1}}\left( {{\Delta _{Z,h}}u\left( \xi  \right)} \right) - {J_{\alpha  + 1}}\left( {{\Delta _{Z,h}}\left( \eta  \right)} \right)} \right)\frac{{\left( {{\varphi ^p}\left( \xi  \right) + {\varphi ^p}\left( \eta  \right)} \right)}}{2}\nonumber\\
  & - \left| {{J_p}\left( {u\left( {\xi {e^{hZ}}} \right) - u\left( {\eta {e^{hZ}}} \right)} \right) - {J_p}\left( {u\left( \xi  \right) - u\left( \eta  \right)} \right)} \right|\left( {{{\left| {{\Delta _{Z,h}}u\left( \xi  \right)} \right|}^\alpha } + {{\left| {{\Delta _{Z,h}}\left( \eta  \right)} \right|}^\alpha }} \right)\left| {\frac{{{\varphi ^p}\left( \xi  \right) - {\varphi ^p}\left( \eta  \right)}}{2}} \right|\nonumber\\
  : = &{{J'}_{11}} - {{J'}_{12}}.
 \end{align}

\textbf{Estimate of ${J'_{11}}$.} We get from \eqref{eqA5} that
\begin{equation}\label{eq453}
 {{J'}_{11}} \ge \frac{1}{{c\left( {p,\alpha } \right)}}{\left| {{{\left| {{\Delta _{Z,h}}u\left( \xi  \right)} \right|}^{\frac{{\alpha  - 1}}{p}}}\left( {{\Delta _{Z,h}}u\left( \xi  \right)} \right) - {{\left| {{\Delta _{Z,h}}\left( \eta  \right)} \right|}^{\frac{{\alpha  - 1}}{p}}}\left( {{\Delta _{Z,h}}\left( \eta  \right)} \right)} \right|^p}\left( {{\varphi ^p}\left( \xi  \right) + {\varphi ^p}\left( \eta  \right)} \right).
\end{equation}

\textbf{Estimate of ${J'_{12}}$.} By using \eqref{eqA4}, Young's inequality and \eqref{eqA6}, it yields
\begin{align}\label{eq451}
   {J'_{12}} =& \left| {{J_p}\left( {u\left( {\xi {e^{hZ}}} \right) - u\left( {\eta {e^{hZ}}} \right)} \right) - {J_p}\left( {u\left( \xi  \right) - u\left( \eta  \right)} \right)} \right| \nonumber\\
   &  \times \left( {{{\left| {{\Delta _{Z,h}}u\left( \xi  \right)} \right|}^\alpha } + {{\left| {{\Delta _{Z,h}}\left( \eta  \right)} \right|}^\alpha }} \right)\left| {\frac{{{\varphi ^p}\left( \xi  \right) - {\varphi ^p}\left( \eta  \right)}}{2}} \right|\nonumber\\
    \le& \frac{{2\left( {p - 1} \right)}}{p}\left( {{{\left| {u\left( {\xi {e^{hZ}}} \right) - u\left( {\eta {e^{hZ}}} \right)} \right|}^{\frac{{p - 2}}{2}}} + {{\left| {u\left( \xi  \right) - u\left( \eta  \right)} \right|}^{\frac{{p - 2}}{2}}}} \right)\nonumber\\
    & \times \left| {{{\left| {u\left( {\xi {e^{hZ}}} \right) - u\left( {\eta {e^{hZ}}} \right)} \right|}^{\frac{{p - 2}}{2}}}\left( {u\left( {\xi {e^{hZ}}} \right) - u\left( {\eta {e^{hZ}}} \right)} \right) - {{\left| {u\left( \xi  \right) - u\left( \eta  \right)} \right|}^{\frac{{p - 2}}{2}}}\left( {u\left( \xi  \right) - u\left( \eta  \right)} \right)} \right|\nonumber\\
    & \times \left( {{{\left| {{\Delta _{Z,h}}u\left( \xi  \right)} \right|}^\alpha } + {{\left| {{\Delta _{Z,h}}\left( \eta  \right)} \right|}^\alpha }} \right)\frac{{{\varphi ^{\frac{p}{2}}}\left( \xi  \right) + {\varphi ^{\frac{p}{2}}}\left( \eta  \right)}}{2}\left| {{\varphi ^{\frac{p}{2}}}\left( \xi  \right) - {\varphi ^{\frac{p}{2}}}\left( \eta  \right)} \right|\nonumber\\
     \le &\frac{c}{\varepsilon }{\left( {{{\left| {u\left( {\xi {e^{hZ}}} \right) - u\left( {\eta {e^{hZ}}} \right)} \right|}^{\frac{{p - 2}}{2}}} + {{\left| {u\left( \xi  \right) - u\left( \eta  \right)} \right|}^{\frac{{p - 2}}{2}}}} \right)^2}\nonumber\\
     & \times \left( {{{\left| {{\Delta _{Z,h}}u\left( \xi  \right)} \right|}^{\alpha  + 1}} + {{\left| {{\Delta _{Z,h}}\left( \eta  \right)} \right|}^{\alpha  + 1}}} \right){\left| {{\varphi ^{\frac{p}{2}}}\left( \xi  \right) - {\varphi ^{\frac{p}{2}}}\left( \eta  \right)} \right|^2}\nonumber\\
     & + c\varepsilon {\left| {{{\left| {u\left( {\xi {e^{hZ}}} \right) - u\left( {\eta {e^{hZ}}} \right)} \right|}^{\frac{{p - 2}}{2}}}\left( {u\left( {\xi {e^{hZ}}} \right) - u\left( {\eta {e^{hZ}}} \right)} \right) - {{\left| {u\left( \xi  \right) - u\left( \eta  \right)} \right|}^{\frac{{p - 2}}{2}}}\left( {u\left( \xi  \right) - u\left( \eta  \right)} \right)} \right|^2}\nonumber\\
     & \times \left( {{{\left| {{\Delta _{Z,h}}u\left( \xi  \right)} \right|}^{\alpha  - 1}} + {{\left| {{\Delta _{Z,h}}\left( \eta  \right)} \right|}^{\alpha  - 1}}} \right)\left( {{\varphi ^p}\left( \xi  \right) + {\varphi ^p}\left( \eta  \right)} \right)\nonumber\\
      \le& \frac{c}{\varepsilon }{\left( {{{\left| {u\left( {\xi {e^{hZ}}} \right) - u\left( {\eta {e^{hZ}}} \right)} \right|}^{\frac{{p - 2}}{2}}} + {{\left| {u\left( \xi  \right) - u\left( \eta  \right)} \right|}^{\frac{{p - 2}}{2}}}} \right)^2}\nonumber\\
      & \times \left( {{{\left| {{\Delta _{Z,h}}u\left( \xi  \right)} \right|}^{\alpha  + 1}} + {{\left| {{\Delta _{Z,h}}\left( \eta  \right)} \right|}^{\alpha  + 1}}} \right){\left| {{\varphi ^{\frac{p}{2}}}\left( \xi  \right) - {\varphi ^{\frac{p}{2}}}\left( \eta  \right)} \right|^2}\nonumber\\
     & + c\varepsilon \left( {{J_p}\left( {u\left( {\xi {e^{hZ}}} \right) - u\left( {\eta {e^{hZ}}} \right)} \right) - {J_p}\left( {u\left( \xi  \right) - u\left( \eta  \right)} \right)} \right)\nonumber\\
     & \times \left( {{J_{\alpha  + 1}}\left( {{\Delta _{Z,h}}u\left( \xi  \right)} \right) - {J_{\alpha  + 1}}\left( {{\Delta _{Z,h}}\left( \eta  \right)} \right)} \right)\left( {{\varphi ^p}\left( \xi  \right) + {\varphi ^p}\left( \eta  \right)} \right)\nonumber\\
     =& \frac{c}{\varepsilon }{\left( {{{\left| {u\left( {\xi {e^{hZ}}} \right) - u\left( {\eta {e^{hZ}}} \right)} \right|}^{\frac{{p - 2}}{2}}} + {{\left| {u\left( \xi  \right) - u\left( \eta  \right)} \right|}^{\frac{{p - 2}}{2}}}} \right)^2}\nonumber\\
      & \times \left( {{{\left| {{\Delta _{Z,h}}u\left( \xi  \right)} \right|}^{\alpha  + 1}} + {{\left| {{\Delta _{Z,h}}\left( \eta  \right)} \right|}^{\alpha  + 1}}} \right){\left| {{\varphi ^{\frac{p}{2}}}\left( \xi  \right) - {\varphi ^{\frac{p}{2}}}\left( \eta  \right)} \right|^2} + 2c \varepsilon {J'_{11}},
\end{align}
where $c=c(p)>0$ and $\varepsilon$ is arbitrary. Substituting \eqref{eq451} into \eqref{eq450} and taking $\varepsilon$ sufficiently small, then we obtain from \eqref{eq453}
\begin{align}\label{eq452}
  J'_1 \ge& \frac{1}{c(p)}{{J'}_{11}} - c(p){\left( {{{\left| {u\left( {\xi {e^{hZ}}} \right) - u\left( {\eta {e^{hZ}}} \right)} \right|}^{\frac{{p - 2}}{2}}} + {{\left| {u\left( \xi  \right) - u\left( \eta  \right)} \right|}^{\frac{{p - 2}}{2}}}} \right)^2}\nonumber\\
      & \times \left( {{{\left| {{\Delta _{Z,h}}u\left( \xi  \right)} \right|}^{\alpha  + 1}} + {{\left| {{\Delta _{Z,h}}\left( \eta  \right)} \right|}^{\alpha  + 1}}} \right){\left| {{\varphi ^{\frac{p}{2}}}\left( \xi  \right) - {\varphi ^{\frac{p}{2}}}\left( \eta  \right)} \right|^2}\nonumber\\
   \ge&c(p,\alpha){\left| {{{\left| {{\Delta _{Z,h}}u\left( \xi  \right)} \right|}^{\frac{{\alpha  - 1}}{p}}}\left( {{\Delta _{Z,h}}u\left( \xi  \right)} \right) - {{\left| {{\Delta _{Z,h}}\left( \eta  \right)} \right|}^{\frac{{\alpha  - 1}}{p}}}\left( {{\Delta _{Z,h}}\left( \eta  \right)} \right)} \right|^p}\left( {{\varphi ^p}\left( \xi  \right) + {\varphi ^p}\left( \eta  \right)} \right)\nonumber\\
  & - c(p){\left( {{{\left| {u\left( {\xi {e^{hZ}}} \right) - u\left( {\eta {e^{hZ}}} \right)} \right|}^{\frac{{p - 2}}{2}}} + {{\left| {u\left( \xi  \right) - u\left( \eta  \right)} \right|}^{\frac{{p - 2}}{2}}}} \right)^2}\nonumber\\
      & \times \left( {{{\left| {{\Delta _{Z,h}}u\left( \xi  \right)} \right|}^{\alpha  + 1}} + {{\left| {{\Delta _{Z,h}}\left( \eta  \right)} \right|}^{\alpha  + 1}}} \right){\left| {{\varphi ^{\frac{p}{2}}}\left( \xi  \right) - {\varphi ^{\frac{p}{2}}}\left( \eta  \right)} \right|^2}.
\end{align}
Therefore, it gets
\begin{align}\label{eq454}
   {J_1} =& \int_{{B_R}} {\int_{{B_R}} {\frac{{{{J'}_1}}}{{{{\left| h \right|}^{1 + \theta \alpha }}}}d\mu } } \nonumber \\
   \ge& c\left( {p,\alpha } \right)\int_{{B_R}} {\int_{{B_R}} {{{\left| {\frac{{{{\left| {{\Delta _{Z,h}}u\left( \xi  \right)} \right|}^{\frac{{\alpha  - 1}}{p}}}\left( {{\Delta _{Z,h}}u\left( \xi  \right)} \right)}}{{{{\left| h \right|}^{\frac{{1 + \theta \alpha }}{p}}}}} - \frac{{{{\left| {{\Delta _{Z,h}}\left( \eta  \right)} \right|}^{\frac{{\alpha  - 1}}{p}}}\left( {{\Delta _{Z,h}}\left( \eta  \right)} \right)}}{{{{\left| h \right|}^{\frac{{1 + \theta \alpha }}{p}}}}}} \right|}^p}\left( {{\varphi ^p}\left( \xi  \right) + {\varphi ^p}\left( \eta  \right)} \right)d\mu } } \nonumber\\
   & - c\left( p \right)\int_{{B_R}} {\int_{{B_R}} {{{\left( {{{\left| {u\left( {\xi {e^{hZ}}} \right) - u\left( {\eta {e^{hZ}}} \right)} \right|}^{\frac{{p - 2}}{2}}} + {{\left| {u\left( \xi  \right) - u\left( \eta  \right)} \right|}^{\frac{{p - 2}}{2}}}} \right)}^2}{{\left| {{\varphi ^{\frac{p}{2}}}\left( \xi  \right) - {\varphi ^{\frac{p}{2}}}\left( \eta  \right)} \right|}^2}} }\nonumber \\
   & \times \frac{{{{\left| {{\Delta _{Z,h}}u\left( \xi  \right)} \right|}^{\alpha  + 1}} + {{\left| {{\Delta _{Z,h}}\left( \eta  \right)} \right|}^{\alpha  + 1}}}}{{{{\left| h \right|}^{1 + \theta \alpha }}}}d\mu .
\end{align}
If we set
\[A = \frac{{{{\left| {{\Delta _{Z,h}}u\left( \xi  \right)} \right|}^{\frac{{\alpha  - 1}}{p}}}\left( {{\Delta _{Z,h}}u\left( \xi  \right)} \right)}}{{{{\left| h \right|}^{\frac{{1 + \theta \alpha }}{p}}}}},\;B = \frac{{{{\left| {{\Delta _{Z,h}}\left( \eta  \right)} \right|}^{\frac{{\alpha  - 1}}{p}}}\left( {{\Delta _{Z,h}}\left( \eta  \right)} \right)}}{{{{\left| h \right|}^{\frac{{1 + \theta \alpha }}{p}}}}}\]
and then use the convexity of $\tau  \mapsto {\tau ^p}$, we have
\begin{align*}
   & {\left| {A\varphi \left( \xi  \right) - B\varphi \left( \eta  \right)} \right|^p} \\
  = &  {\left| {\left( {A - B} \right)\frac{{\varphi \left( \xi  \right) + \varphi \left( \eta  \right)}}{2} + \left( {A + B} \right)\frac{{\varphi \left( \xi  \right) - \varphi \left( \eta  \right)}}{2}} \right|^p}\\
   \le& \frac{1}{2}{\left| {A - B} \right|^p}{\left| {\varphi \left( \xi  \right) + \varphi \left( \eta  \right)} \right|^p} + \frac{1}{2}{\left| {A + B} \right|^p}{\left| {\varphi \left( \xi  \right) - \varphi \left( \eta  \right)} \right|^p}\\
    \le& {2^{p - 2}}{\left| {A - B} \right|^p}\left( {{\varphi ^p}\left( \xi  \right) + {\varphi ^p}\left( \eta  \right)} \right) + {2^{p - 2}}\left( {{{\left| A \right|}^p} + {{\left| B \right|}^p}} \right){\left| {\varphi \left( \xi  \right) - \varphi \left( \eta  \right)} \right|^p}.
\end{align*}
From the above formula and \eqref{eq454}, we have
\begin{align}\label{eq455}
    {J_1} \ge& c\left( {p,\alpha } \right)\int_{{B_R}} {\int_{{B_R}} {{{\left| {\frac{{{{\left| {{\Delta _{Z,h}}u\left( \xi  \right)} \right|}^{\frac{{\alpha  - 1}}{p}}}\left( {{\Delta _{Z,h}}u\left( \xi  \right)} \right)}}{{{{\left| h \right|}^{\frac{{1 + \theta \alpha }}{p}}}}}\varphi \left( \xi  \right) - \frac{{{{\left| {{\Delta _{Z,h}}\left( \eta  \right)} \right|}^{\frac{{\alpha  - 1}}{p}}}\left( {{\Delta _{Z,h}}\left( \eta  \right)} \right)}}{{{{\left| h \right|}^{\frac{{1 + \theta \alpha }}{p}}}}}\varphi \left( \eta  \right)} \right|}^p}d\mu } }  \nonumber \\
   &- c\left( p \right)\int_{{B_R}} {\int_{{B_R}} {{{\left( {{{\left| {u\left( {\xi {e^{hZ}}} \right) - u\left( {\eta {e^{hZ}}} \right)} \right|}^{\frac{{p - 2}}{2}}} + {{\left| {u\left( \xi  \right) - u\left( \eta  \right)} \right|}^{\frac{{p - 2}}{2}}}} \right)}^2}{{\left| {{\varphi ^{\frac{p}{2}}}\left( \xi  \right) - {\varphi ^{\frac{p}{2}}}\left( \eta  \right)} \right|}^2}} } \nonumber \\
   & \times \frac{{{{\left| {{\Delta _{Z,h}}u\left( \xi  \right)} \right|}^{\alpha  + 1}} + {{\left| {{\Delta _{Z,h}}\left( \eta  \right)} \right|}^{\alpha  + 1}}}}{{{{\left| h \right|}^{1 + \theta \alpha }}}}d\mu  \nonumber \\
  &  - c\left( {p,\alpha } \right)\int_{{B_R}} {\int_{{B_R}} {\left( {\frac{{{{\left| {{\Delta _{Z,h}}u\left( \xi  \right)} \right|}^{\alpha  - 1 + p}}}}{{{{\left| h \right|}^{1 + \theta \alpha }}}} + \frac{{{{\left| {{\Delta _{Z,h}}\left( \eta  \right)} \right|}^{\alpha  - 1 + p}}}}{{{{\left| h \right|}^{1 + \theta \alpha }}}}} \right){{\left| {\varphi \left( \xi  \right) - \varphi \left( \eta  \right)} \right|}^p}d\mu } } \nonumber \\
   : =& c\left( {p,\alpha } \right)\left[ {\frac{{{{\left| {{\Delta _{Z,h}}} \right|}^{\frac{{\alpha  - 1}}{p}}}{\Delta _{Z,h}}}}{{{{\left| h \right|}^{\frac{{1 + \theta \alpha }}{p}}}}}\varphi } \right]_{H{W^{s,p}}\left( {{B_R}} \right)}^p - c\left( {p } \right){J_{11}} - c\left( {p,\alpha } \right){J_{12}},
\end{align}
where
\begin{align*}
   J_{11}=& \int_{{B_R}} {\int_{{B_R}} {{{\left( {{{\left| {u\left( {\xi {e^{hZ}}} \right) - u\left( {\eta {e^{hZ}}} \right)} \right|}^{\frac{{p - 2}}{2}}} + {{\left| {u\left( \xi  \right) - u\left( \eta  \right)} \right|}^{\frac{{p - 2}}{2}}}} \right)}^2}{{\left| {{\varphi ^{\frac{p}{2}}}\left( \xi  \right) - {\varphi ^{\frac{p}{2}}}\left( \eta  \right)} \right|}^2}} } \nonumber \\
   & \times \frac{{{{\left| {{\Delta _{Z,h}}u\left( \xi  \right)} \right|}^{\alpha  + 1}} + {{\left| {{\Delta _{Z,h}}\left( \eta  \right)} \right|}^{\alpha  + 1}}}}{{{{\left| h \right|}^{1 + \theta \alpha }}}}d\mu
\end{align*}
and
\begin{equation*}
{J_{12}}=\int_{{B_R}} {\int_{{B_R}} {\left( {\frac{{{{\left| {{\Delta _{Z,h}}u\left( \xi  \right)} \right|}^{\alpha  - 1 + p}}}}{{{{\left| h \right|}^{1 + \theta \alpha }}}} + \frac{{{{\left| {{\Delta _{Z,h}}\left( \eta  \right)} \right|}^{\alpha  - 1 + p}}}}{{{{\left| h \right|}^{1 + \theta \alpha }}}}} \right){{\left| {\varphi \left( \xi  \right) - \varphi \left( \eta  \right)} \right|}^p}d\mu } }.
\end{equation*}

\textbf{Estimate of ${J_{11}}$.} We first estimate
\[\int_{{B_R}} {\int_{{B_R}} {{{\left| {u\left( \xi  \right) - u\left( \eta  \right)} \right|}^{p - 2}}{{\left| {{\varphi ^{\frac{p}{2}}}\left( \xi  \right) - {\varphi ^{\frac{p}{2}}}\left( \eta  \right)} \right|}^2}\frac{{{{\left| {{\Delta _{Z,h}}u\left( \xi  \right)} \right|}^{\alpha  + 1}}}}{{{{\left| h \right|}^{1 + \theta \alpha }}}}d\mu } } .\]
The other terms of ${J_{11}}$ can be treat similarly. Consider that $\varphi$ is Lipschitz and $p \ge 2$, it yields
\begin{align}\label{eq456}
   & \int_{{B_R}} {\int_{{B_R}} {{{\left| {u\left( \xi  \right) - u\left( \eta  \right)} \right|}^{p - 2}}{{\left| {{\varphi ^{\frac{p}{2}}}\left( \xi  \right) - {\varphi ^{\frac{p}{2}}}\left( \eta  \right)} \right|}^2}\frac{{{{\left| {{\Delta _{Z,h}}u\left( \xi  \right)} \right|}^{\alpha  + 1}}}}{{{{\left| h \right|}^{1 + \theta \alpha }}}}d\mu } } \nonumber \\
   \le&  \frac{c}{{h_0^2}}\int_{{B_R}} {\int_{{B_R}} {\frac{{{{\left| {u\left( \xi  \right) - u\left( \eta  \right)} \right|}^{p - 2}}}}{{\left\| {{\eta ^{ - 1}} \circ \xi } \right\|_{{{\mathbb{H}}^n}}^{Q + sp - 2}}}\frac{{{{\left| {{\Delta _{Z,h}}u\left( \xi  \right)} \right|}^{\alpha  + 1}}}}{{{{\left| h \right|}^{1 + \theta \alpha }}}}d\xi d\eta } } .
\end{align}

For $p=2$, the last term of $\eqref{eq456}$ reduces to
\begin{align}\label{eq457}
   & \int_{{B_R}} {\int_{{B_R}} {\frac{1}{{\left\| {{\eta ^{ - 1}} \circ \xi } \right\|_{{{\mathbb{H}}^n}}^{Q + 2s - 2}}}\frac{{{{\left| {{\Delta _{Z,h}}u\left( \xi  \right)} \right|}^{\alpha  + 1}}}}{{{{\left| h \right|}^{1 + \theta \alpha }}}}d\xi d\eta } }  \nonumber\\
   =&  \int_{{B_R}} {\left( {\int_{{B_R}} {\frac{1}{{\left\| {{\eta ^{ - 1}} \circ \xi } \right\|_{{{\mathbb{H}}^n}}^{Q + 2s - 2}}}d\eta } } \right)} \frac{{{{\left| {{\Delta _{Z,h}}u\left( \xi  \right)} \right|}^{\alpha  + 1}}}}{{{{\left| h \right|}^{1 + \theta \alpha }}}}d\xi \nonumber\\
   \le& c\int_{{B_R}} {\frac{{{{\left| {{\Delta _{Z,h}}u\left( \xi  \right)} \right|}^{\alpha  + 1}}}}{{{{\left| h \right|}^{1 + \theta \alpha }}}}d\xi }  \le c{\left\| u \right\|_{{L^\infty }\left( {{B_1}} \right)}}\int_{{B_R}} {\frac{{{{\left| {{\Delta _{Z,h}}u\left( \xi  \right)} \right|}^\alpha }}}{{{{\left| h \right|}^{1 + \theta \alpha }}}}d\xi }  \nonumber\\
   \le& c\int_{{B_R}} {\frac{{{{\left| {{\Delta _{Z,h}}u\left( \xi  \right)} \right|}^\alpha }}}{{{{\left| h \right|}^{1 + \theta \alpha }}}}d\xi },
\end{align}
where we used ${\left\| u \right\|_{{L^\infty }\left( {{B_1}} \right)}} \le 1$, and $c=c(Q,s)$.

For $p >2$, we take
\[\varepsilon  = \min \left\{ {\frac{{p - 2}}{2},\frac{1}{s} - 1} \right\} > 0,\]
then
\begin{equation}\label{eq459}
\frac{{q\left( {2 - 2s - \varepsilon s} \right)}}{{q - p + 2}} > 0.
\end{equation}
Thus, by Young's inequality with exponents $\frac{q}{{p - 2}}$ and $\frac{q}{{q - p + 2}}$, we get
\begin{align}\label{eq458}
   & \frac{c}{{h_0^2}}\int_{{B_R}} {\int_{{B_R}} {\frac{{{{\left| {u\left( \xi  \right) - u\left( \eta  \right)} \right|}^{p - 2}}}}{{\left\| {{\eta ^{ - 1}} \circ \xi } \right\|_{{{\mathbb{H}}^n}}^{Q + sp - 2}}}\frac{{{{\left| {{\Delta _{Z,h}}u\left( \xi  \right)} \right|}^{\alpha  + 1}}}}{{{{\left| h \right|}^{1 + \theta \alpha }}}}d\xi d\eta } }  \nonumber\\
   \le&  c\int_{{B_R}} {\int_{{B_R}} {\frac{{{{\left| {u\left( \xi  \right) - u\left( \eta  \right)} \right|}^q}}}{{\left\| {{\eta ^{ - 1}} \circ \xi } \right\|_{{{\mathbb{H}}^n}}^{Q + \frac{{s\left( {p - 2 - \varepsilon } \right)}}{{p - 2}}q}}}d\xi d\eta } }\nonumber \\
   & + {\left( {\frac{c}{{{h_0}}}} \right)^{\frac{{2q}}{{q - p + 2}}}}\int_{{B_R}} {\left( {\int_{{B_R}} {\left\| {{\eta ^{ - 1}} \circ \xi } \right\|_{{{\mathbb{H}}^n}}^{\frac{{q\left( {2 - 2s - \varepsilon s} \right)}}{{q - p + 2}} - Q}d\eta } } \right)} \frac{{{{\left| {{\Delta _{Z,h}}u\left( \xi  \right)} \right|}^{\frac{{q\left( {\alpha  + 1} \right)}}{{q - p + 2}}}}}}{{{{\left| h \right|}^{\frac{{q\left( {1 + \theta \alpha } \right)}}{{q - p + 2}}}}}}d\xi \nonumber\\
   \le &c\left[ u \right]_{H{W^{\frac{{s\left( {p - 2 - \varepsilon } \right)}}{{p - 2}},q}}\left( {{B_R}} \right)}^q + c{\left( {\frac{1}{{{h_0}}}} \right)^{\frac{{2q}}{{q - p + 2}}}}\frac{{q - p + 2}}{{q\left( {2 - 2s - \varepsilon s} \right)}}{R^{\frac{{q\left( {2 - 2s - \varepsilon s} \right)}}{{q - p + 2}}}}\int_{{B_R}} {\frac{{{{\left| {{\Delta _{Z,h}}u\left( \xi  \right)} \right|}^{\frac{{q\left( {\alpha  + 1} \right)}}{{q - p + 2}}}}}}{{{{\left| h \right|}^{\frac{{q\left( {1 + \theta \alpha } \right)}}{{q - p + 2}}}}}}d\xi } \nonumber\\
    \le &c\left[ u \right]_{H{W^{\frac{{s\left( {p - 2 - \varepsilon } \right)}}{{p - 2}},q}}\left( {{B_{R + {h_0}}}} \right)}^q + c\left\| u \right\|_{{L^\infty }\left( {{B_1}} \right)}^{\frac{q}{{q - p + 2}}}\int_{{B_R}} {\frac{{{{\left| {{\Delta _{Z,h}}u\left( \xi  \right)} \right|}^{\frac{{\alpha q}}{{q - p + 2}}}}}}{{{{\left| h \right|}^{\frac{{q\left( {1 + \theta \alpha } \right)}}{{q - p + 2}}}}}}d\xi }\nonumber\\
     \le &c\left[ u \right]_{H{W^{\frac{{s\left( {p - 2 - \varepsilon } \right)}}{{p - 2}},q}}\left( {{B_{R + {h_0}}}} \right)}^q + c\int_{{B_R}} {\frac{{{{\left| {{\Delta _{Z,h}}u\left( \xi  \right)} \right|}^{\frac{{\alpha q}}{{q - p + 2}}}}}}{{{{\left| h \right|}^{\frac{{q\left( {1 + \theta \alpha } \right)}}{{q - p + 2}}}}}}d\xi } ,
\end{align}
where $c=c(Q,h_0,p,s,q)>0$ and we also used $q\ge p$ and \eqref{eq459}. Next, by using Lemma \ref{Le26} with
\[\alpha  = \frac{{s\left( {p - 2 - \varepsilon } \right)}}{{p - 2}},\;\beta  = s,\;r = R + 4{h_0},\;\rho  = R + {h_0},\]
and
\[R + 4{h_0} \le 1,\;{\left\| u \right\|_{{L^\infty }\left( {{B_1}} \right)}} \le 1,\]
we have
\begin{equation}\label{eq473}
\left[ u \right]_{H{W^{\frac{{s\left( {p - 2 - \varepsilon } \right)}}{{p - 2}},q}}\left( {{B_{R + {h_0}}}} \right)}^q \le c\left( {\mathop {\sup }\limits_{0 < \left| h \right| < {h_0}} \left\| {\frac{{\Delta _{Z,h}^2u}}{{{{\left| h \right|}^s}}}} \right\|_{{L^q}\left( {{B_{R + 4{h_0}}}} \right)}^q + 1} \right).
\end{equation}
Therefore, for $p \ge 2$, it follows
\begin{equation}\label{eq466}
\left| {{J_{11}}} \right| \le c\left( {\int_{{B_R}} {\frac{{{{\left| {{\Delta _{Z,h}}u\left( \xi  \right)} \right|}^{\frac{{\alpha q}}{{q - p + 2}}}}}}{{{{\left| h \right|}^{\frac{{q\left( {1 + \theta \alpha } \right)}}{{q - p + 2}}}}}}d\xi }  + \mathop {\sup }\limits_{0 < \left| h \right| < {h_0}} \left\| {\frac{{\Delta _{Z,h}^2u}}{{{{\left| h \right|}^s}}}} \right\|_{{L^q}\left( {{B_{R + 4{h_0}}}} \right)}^q + 1} \right),
\end{equation}
where $c=c(Q,h_0,p,s,q)>0$.

\textbf{Estimate of ${J_{12}}$.} By means of Lipschitz continuous of $\varphi $, the Young inequality with exponents $\frac{q}{{p - 2}}$ and $\frac{q}{{q - p + 2}}$ and ${\left\| u \right\|_{{L^\infty }\left( {{B_1}} \right)}} \le 1$, we deduce
\begin{align*}
   {J_{12}}& = \int_{{B_R}} {\int_{{B_R}} {\left( {\frac{{{{\left| {{\Delta _{Z,h}}u\left( \xi  \right)} \right|}^{\alpha  - 1 + p}}}}{{{{\left| h \right|}^{1 + \theta \alpha }}}} + \frac{{{{\left| {{\Delta _{Z,h}}\left( \eta  \right)} \right|}^{\alpha  - 1 + p}}}}{{{{\left| h \right|}^{1 + \theta \alpha }}}}} \right){{\left| {\varphi \left( \xi  \right) - \varphi \left( \eta  \right)} \right|}^p}d\mu } }  \\
   &  = 2\int_{{B_R}} {\int_{{B_R}} {\frac{{{{\left| {{\Delta _{Z,h}}u\left( \xi  \right)} \right|}^{\alpha  - 1 + p}}}}{{{{\left| h \right|}^{1 + \theta \alpha }}}}\frac{{{{\left| {\varphi \left( \xi  \right) - \varphi \left( \eta  \right)} \right|}^p}}}{{\left\| {{\eta ^{ - 1}} \circ \xi } \right\|_{{{\mathbb{H}}^n}}^{Q + sp}}}d\xi d\eta } } \\
  & \le c\int_{{B_R}} {\int_{{B_R}} {\frac{{{{\left| {{\Delta _{Z,h}}u\left( \xi  \right)} \right|}^{\alpha  - 1 + p}}}}{{{{\left| h \right|}^{1 + \theta \alpha }}}}\frac{1}{{\left\| {{\eta ^{ - 1}} \circ \xi } \right\|_{{{\mathbb{H}}^n}}^{Q + \left( {s - 1} \right)p}}}d\xi d\eta } } \\
 & \le c\int_{{B_R}} {\frac{{{{\left| {{\Delta _{Z,h}}u\left( \xi  \right)} \right|}^{\alpha  - 1 + p}}}}{{{{\left| h \right|}^{1 + \theta \alpha }}}}d\xi } \\
 & \le c\left( {\int_{{B_R}} {\frac{{{{\left| {{\Delta _{Z,h}}u\left( \xi  \right)} \right|}^{\alpha \frac{q}{{q - p + 2}}}}}}{{{{\left| h \right|}^{\left( {1 + \theta \alpha } \right)\frac{q}{{q - p + 2}}}}}}d\xi  + \int_{{B_R}} {{{\left| {{\Delta _{Z,h}}u\left( \xi  \right)} \right|}^{\left( {p - 1} \right)\frac{q}{{p - 2}}}}d\xi } } } \right)\\
 & \le c\left( {\int_{{B_R}} {\frac{{{{\left| {{\Delta _{Z,h}}u\left( \xi  \right)} \right|}^{\alpha \frac{q}{{q - p + 2}}}}}}{{{{\left| h \right|}^{\left( {1 + \theta \alpha } \right)\frac{q}{{q - p + 2}}}}}}d\xi  + 1} } \right),
\end{align*}
i.e.
\begin{equation}\label{eq467}
\left| {{J_{12}}} \right| \le c\left( {\int_{{B_R}} {\frac{{{{\left| {{\Delta _{Z,h}}u\left( \xi  \right)} \right|}^{\alpha \frac{q}{{q - p + 2}}}}}}{{{{\left| h \right|}^{\left( {1 + \theta \alpha } \right)\frac{q}{{q - p + 2}}}}}}d\xi  + 1} } \right),
\end{equation}
where $c=c(Q,h_0,p,s,q)>0$.

Therefore, substituting \eqref{eq466} and \eqref{eq467} into \eqref{eq455}, we obtain
\begin{align}\label{eq468}
 {J_1} \ge& c\left( {p,\alpha } \right)\left[ {\frac{{{{\left| {{\Delta _{Z,h}}} \right|}^{\frac{{\alpha  - 1}}{p}}}{\Delta _{Z,h}}}}{{{{\left| h \right|}^{\frac{{1 + \theta \alpha }}{p}}}}}\varphi } \right]_{H{W^{s,p}}\left( {{B_R}} \right)}^p \nonumber\\
 &- c\left( {\int_{{B_R}} {\frac{{{{\left| {{\Delta _{Z,h}}u\left( \xi  \right)} \right|}^{\frac{{\alpha q}}{{q - p + 2}}}}}}{{{{\left| h \right|}^{\frac{{q\left( {1 + \theta \alpha } \right)}}{{q - p + 2}}}}}}d\xi }  + \mathop {\sup }\limits_{0 < \left| h \right| < {h_0}} \left\| {\frac{{\Delta _{Z,h}^2u}}{{{{\left| h \right|}^s}}}} \right\|_{{L^q}\left( {{B_{R + 4{h_0}}}} \right)}^q + 1} \right),
\end{align}
where $c=c(Q,h_0,p,s,q)>0$.

\textbf{Estimate of ${J_{2}}$ and ${J_{3}}$.} Because both nonlocal ${J_{2}}$ and ${J_{3}}$ can be treated in the same way, we only estimate ${J_{2}}$. It gets from the boundedness of $u$ that
\begin{align*}\label{eq469}
   & \left| {\left( {{J_p}\left( {u\left( {\xi {e^{hZ}}} \right) - u\left( {\eta {e^{hZ}}} \right)} \right) - {J_p}\left( {u\left( \xi  \right) - u\left( \eta  \right)} \right)} \right){J_{\alpha  + 1}}\left( {{\Delta _{Z,h}}u\left( \xi  \right)} \right)} \right| \nonumber\\
   \le& c(p)\left( {1 + {{\left| {u\left( {\eta {e^{hZ}}} \right)} \right|}^{p - 1}} + {{\left| {u\left( \eta  \right)} \right|}^{p - 1}}} \right){\left| {{\Delta _{Z,h}}u\left( \xi  \right)} \right|^\alpha }.
\end{align*}
For $\xi  \in {B_{\frac{{R + r}}{2}}},$ we have ${B_{\frac{{R - r}}{2}}}\left( \xi  \right) \subset {B_R}$ and thus
\[\int_{{{\mathbb{H}}^n}\backslash {B_R}} {\frac{1}{{\left\| {{\eta ^{ - 1}} \circ \xi } \right\|_{{{\mathbb{H}}^n}}^{Q + sp}}}d\eta }  \le \int_{{{\mathbb{H}}^n}\backslash {B_{\frac{{R - r}}{2}}}\left( \xi  \right)} {\frac{1}{{\left\| {{\eta ^{ - 1}} \circ \xi } \right\|_{{{\mathbb{H}}^n}}^{Q + sp}}}d\eta }  \le c\left( {Q,{h_0},p,s} \right),\]
by recalling $R - r = 4{h_0}$. For $\xi  \in {B_{\frac{{R + r}}{2}}}$, by using Lemma \ref{Le29} and Lemma \ref{Le210}, we have
\begin{align}\label{eq481}
   & \int_{{{\mathbb{H}}^n}\backslash {B_R}} {\frac{{{{\left| {u\left( \eta  \right)} \right|}^{p - 1}}}}{{\left\| {{\eta ^{ - 1}} \circ \xi } \right\|_{{{\mathbb{H}}^n}}^{Q + sp}}}d\eta }  \le {\left( {\frac{{2R}}{{R - r}}} \right)^{Q + sp}}\int_{{{\mathbb{H}}^n}\backslash {B_R}} {\frac{{{{\left| {u\left( \eta  \right)} \right|}^{p - 1}}}}{{\left\| \eta  \right\|_{{{\mathbb{H}}^n}}^{Q + sp}}}d\eta }\nonumber  \\
   \le& {\left( {\frac{{2R}}{{R - r}}} \right)^{Q + sp}}\int_{{{\mathbb{H}}^n}\backslash {B_1}} {\frac{{{{\left| {u\left( \eta  \right)} \right|}^{p - 1}}}}{{\left\| \eta  \right\|_{{{\mathbb{H}}^n}}^{Q + sp}}}d\eta }  + {\left( {\frac{{2R}}{{R - r}}} \right)^{Q + sp}}{R^{ - Q-sp}}\int_{{B_1}} {{{\left| {u\left( \eta  \right)} \right|}^{p - 1}}d\eta } \nonumber\\
    \le& c\left( {Q,{h_0},p,s} \right),
\end{align}
where in the last estimate we have used the bounds assumed on $u$ and $4h_0<R<1$. The term involving ${u\left( {\eta {e^{hZ}}} \right)}$ can be estimated similarly. Recall that $\varphi  = 0$ outside ${B_{\frac{{R + r}}{2}}}$, we use the Young inequality to get
\begin{align}\label{eq470}
   \left| {{J_2}} \right| + \left| {{J_3}} \right|& \le c\left( {Q,{h_0},p,s} \right)\int_{{B_{\frac{{R + r}}{2}}}} {\frac{{{{\left| {{\Delta _{Z,h}}u} \right|}^\alpha }}}{{{{\left| h \right|}^{1 + \theta \alpha }}}}d\xi } \nonumber \\
   &  \le c\left( {Q,{h_0},p,q,s} \right)\left( {1 + \int_{{B_R}} {{{\left( {\frac{{\left| {{\Delta _{Z,h}}u} \right|}}{{{{\left| h \right|}^{\frac{{1 + \theta \alpha }}{\alpha }}}}}} \right)}^{\frac{{\alpha q}}{{q - p + 2}}}}d\xi } } \right).
\end{align}

Therefore, substituting \eqref{eq468} and \eqref{eq470} into \eqref{eq423}, we get
\begin{align}\label{eq471}
  J \ge &c\left( {p,\alpha } \right)\left[ {\frac{{{{\left| {{\Delta _{Z,h}}} \right|}^{\frac{{\alpha  - 1}}{p}}}{\Delta _{Z,h}}}}{{{{\left| h \right|}^{\frac{{1 + \theta \alpha }}{p}}}}}\varphi } \right]_{H{W^{s,p}}\left( {{B_R}} \right)}^p \nonumber\\
  &- c\left( {\int_{{B_R}} {{{\left( {\frac{{\left| {{\Delta _{Z,h}}u} \right|}}{{{{\left| h \right|}^{\frac{{1 + \theta \alpha }}{\alpha }}}}}} \right)}^{\frac{{\alpha q}}{{q - p + 2}}}}d\xi }  + \mathop {\sup }\limits_{0 < \left| h \right| < {h_0}} \left\| {\frac{{\Delta _{Z,h}^2u}}{{{{\left| h \right|}^s}}}} \right\|_{{L^q}\left( {{B_{R + 4{h_0}}}} \right)}^q + 1} \right),
\end{align}
where $c=c(Q,h_0,p,s,q,\alpha)>0$.

\textbf{Step 4: Going back to the equation.} Inserting the estimates \eqref{eq422} and \eqref{eq471} in \eqref{eq46}, then we use \eqref{eq473} to get
\begin{align}\label{eq472}
   & \int_{{B_R}} {{{\left| {{\nabla _H}\left( {\frac{{{{\left| {u\left( {\xi {e^{hZ}}} \right) - u} \right|}^{\frac{{\alpha  - 1}}{p}}}\left( {u\left( {\xi {e^{hZ}}} \right) - u} \right)\varphi }}{{{{\left| h \right|}^{\frac{{1 + \theta \alpha }}{p}}}}}} \right)} \right|}^p}d\xi } +\left[ {\frac{{{{\left| {{\Delta _{Z,h}}} \right|}^{\frac{{\alpha  - 1}}{p}}}{\Delta _{Z,h}}}}{{{{\left| h \right|}^{\frac{{1 + \theta \alpha }}{p}}}}}\varphi } \right]_{H{W^{s,p}}\left( {{B_R}} \right)}^p \nonumber \\
 \le  &  c\left( {\int_{{B_R}} {{{\left| {{\nabla _H}u\left( {\xi {e^{hZ}}} \right)} \right|}^q}d\xi }  + \int_{{B_R}} {{{\left| {{\nabla _H}u} \right|}^q}d\xi }  + \int_{{B_R}} {{{\left( {\frac{{\left| {{\Delta _{Z,h}}u} \right|}}{{{{\left| h \right|}^{\frac{{1 + \theta \alpha }}{\alpha }}}}}} \right)}^{\frac{{\alpha q}}{{q - p + 2}}}}d\xi }  + 1} \right)\nonumber\\
 & + c\Lambda \left( {\int_{{B_R}} {{{\left( {\frac{{\left| {{\Delta _{Z,h}}u} \right|}}{{{{\left| h \right|}^{\frac{{1 + \theta \alpha }}{\alpha }}}}}} \right)}^{\frac{{\alpha q}}{{q - p + 2}}}}d\xi }  + \mathop {\sup }\limits_{0 < \left| h \right| < {h_0}} \left\| {\frac{{\Delta _{Z,h}^2u}}{{{{\left| h \right|}^s}}}} \right\|_{{L^q}\left( {{B_{R + 4{h_0}}}} \right)}^q + 1} \right)\nonumber\\
  \le  &  c\left( {\int_{{B_R}} {{{\left| {{\nabla _H}u\left( {\xi {e^{hZ}}} \right)} \right|}^q}d\xi }  + \int_{{B_R}} {{{\left| {{\nabla _H}u} \right|}^q}d\xi }  + \int_{{B_R}} {{{\left( {\frac{{\left| {{\Delta _{Z,h}}u} \right|}}{{{{\left| h \right|}^{\frac{{1 + \theta \alpha }}{\alpha }}}}}} \right)}^{\frac{{\alpha q}}{{q - p + 2}}}}d\xi }  + 1} \right)\nonumber\\
 & + c\Lambda \left( {\int_{{B_R}} {{{\left( {\frac{{\left| {{\Delta _{Z,h}}u} \right|}}{{{{\left| h \right|}^{\frac{{1 + \theta \alpha }}{\alpha }}}}}} \right)}^{\frac{{\alpha q}}{{q - p + 2}}}}d\xi }  + \int_{{B_{R + {4h_0}}}} {{{\left| {{\nabla _H}u} \right|}^q}d\xi } + 1} \right)\nonumber\\
 \le &c(1+\Lambda) \left( {\int_{{B_R}} {{{\left( {\frac{{\left| {{\Delta _{Z,h}}u} \right|}}{{{{\left| h \right|}^{\frac{{1 + \theta \alpha }}{\alpha }}}}}} \right)}^{\frac{{\alpha q}}{{q - p + 2}}}}d\xi }  + \int_{{B_{R + {4h_0}}}} {{{\left| {{\nabla _H}u} \right|}^q}d\xi } + 1} \right),
\end{align}
where $c=c(Q,h_0,p,s,q,\alpha)>0$ and we also have used
\[\mathop {\sup }\limits_{0 < \left| h \right| < {h_0}} \left\| {\frac{{\Delta _{Z,h}^2u}}{{{{\left| h \right|}^s}}}} \right\|_{{L^q}\left( {{B_{R + 4{h_0}}}} \right)}^q \le h_0^{\left( {1 - s} \right)q}\mathop {\sup }\limits_{0 < \left| h \right| < {h_0}} \left\| {\frac{{\Delta _{Z,h}^2u}}{{\left| h \right|}}} \right\|_{{L^q}\left( {{B_{R + 4{h_0}}}} \right)}^q \le c\left( {{h_0},q} \right)\int_{{B_{R + 4{h_0}}}} {{{\left| {{\nabla _H}u} \right|}^q}d\xi } \]
from Lemma \ref{Le22}.

For $0 < \left| {h'} \right|,\left| h \right| <  {{h_0}} $ we have
\begin{align*}
   &  \varphi {\Delta _{Z,h'}}\left( {{{\left| {{\Delta _{Z,h}}u} \right|}^{\frac{{\alpha  - 1}}{p}}}\left( {{\Delta _{Z,h}}u} \right)} \right)\\
  = & {\Delta _{Z,h'}}\left( {\varphi {{\left| {{\Delta _{Z,h}}u} \right|}^{\frac{{\alpha  - 1}}{p}}}\left( {{\Delta _{Z,h}}u} \right)} \right) - \left( {{\Delta _{Z,h'}}\varphi } \right)\left( {{{\left| {{\Delta _{Z,h}}u} \right|}^{\frac{{\alpha  - 1}}{p}}}\left( {{\Delta _{Z,h}}u} \right)} \right)\left( {\xi  {e^{h'Z}}} \right),
\end{align*}
so it yields from the properties of $\varphi$ and the above formula that
\begin{align}\label{eq474}
   & \left\| {\frac{{{\Delta _{Z,h'}}{\Delta _{Z,h}}u}}{{{{\left| {h'} \right|}^{\frac{{p}}{{\alpha  - 1 + p}}}}{{\left| h \right|}^{\frac{{1 + \theta \alpha }}{{\alpha  - 1 + p}}}}}}} \right\|_{{L^{\alpha  - 1 + p}}\left( {{B_r}} \right)}^{\alpha  - 1 + p}\le c\left\| {\frac{{{\Delta _{Z,h'}}\left( {{{\left| {{\Delta _{Z,h}}u} \right|}^{\frac{{\alpha  - 1}}{p}}}\left( {{\Delta _{Z,h}}u} \right)} \right)}}{{{{\left| {h'} \right|}}{{\left| h \right|}^{\frac{{1 + \theta \alpha }}{p}}}}}} \right\|_{{L^p}\left( {{B_r}} \right)}^p \nonumber\\
  \le & c\left\| {\varphi \frac{{{\Delta _{Z,h'}}\left( {{{\left| {{\Delta _{Z,h}}u} \right|}^{\frac{{\alpha  - 1}}{p}}}\left( {{\Delta _{Z,h}}u} \right)} \right)}}{{{{\left| {h'} \right|}}{{\left| h \right|}^{\frac{{1 + \theta \alpha }}{p}}}}}} \right\|_{{L^p}\left( {{{\mathbb{H}}^n}} \right)}^p\nonumber\\
  \le & c\left\| {\frac{{{\Delta _{Z,h'}}\left( {{{\left| {{\Delta _{Z,h}}u} \right|}^{\frac{{\alpha  - 1}}{p}}}\left( {{\Delta _{Z,h}}u} \right)\varphi} \right)}}{{{{\left| {h'} \right|}}{{\left| h \right|}^{\frac{{1 + \theta \alpha }}{p}}}}}} \right\|_{{L^p}\left( {{{\mathbb{H}}^n}} \right)}^p \nonumber\\
   &+ c\left\| {\frac{{\left( {{\Delta _{Z,h'}}\varphi } \right)\left( {{{\left| {{\Delta _{Z,h}}u} \right|}^{\frac{{\alpha  - 1}}{p}}}\left( {{\Delta _{Z,h}}u} \right)} \right)\left( {\xi   {e^{h'Z}}} \right)}}{{{{\left| {h'} \right|}}{{\left| h \right|}^{\frac{{1 + \theta \alpha }}{p}}}}}} \right\|_{{L^p}\left( {{{\mathbb{H}}^n}} \right)}^p,
\end{align}
where $c=c(p,\alpha)>0$. For the first term in \eqref{eq474}, we apply the Lemma \ref{Le22} and the properties of $\varphi$ to get
\begin{equation}\label{eq475}
\mathop {\sup }\limits_{\left| {h'} \right| > 0} \left\| {\frac{{{\Delta _{Z,h'}}\left( { {{\left| {{\Delta _{Z,h}}u} \right|}^{\frac{{\alpha  - 1}}{p}}}\left( {{\Delta _{Z,h}}u} \right)\varphi} \right)}}{{\left| {h'} \right|{{\left| h \right|}^{\frac{{1 + \theta \alpha }}{p}}}}}} \right\|_{{L^p}\left( {{{\mathbb{H}}^n}} \right)}^p \le c\int_{{B_R}} {{{\left| {{\nabla _H}\left( {\frac{{{{\left| {{\Delta _{Z,h}}u} \right|}^{\frac{{\alpha  - 1}}{p}}}\left( {{\Delta _{Z,h}}u} \right)\varphi }}{{{{\left| h \right|}^{\frac{{1 + \theta \alpha }}{p}}}}}} \right)} \right|}^p}d\xi } ,
\end{equation}
where $c=c(Q,h_0,p)>0$. For the second term in \eqref{eq474}, by means of the properties of $\varphi$, the boundedness of $u$ and Young's inequality with exponents $\frac{q}{{p - 2}}$ and $\frac{q}{{q - p + 2}}$, we derive
\begin{align}\label{eq476}
   & \left\| {\frac{{\left( {{\Delta _{Z,h'}}\varphi } \right)\left( {{{\left| {{\Delta _{Z,h}}u} \right|}^{\frac{{\alpha  - 1}}{p}}}\left( {{\Delta _{Z,h}}u} \right)} \right)\left( {\xi   {e^{h'Z}}} \right)}}{{\left| {h'} \right|{{\left| h \right|}^{\frac{{1 + \theta \alpha }}{p}}}}}} \right\|_{{L^p}\left( {{{\mathbb{H}}^n}} \right)}^p\nonumber \\
   \le&  c\left\| {\frac{{\left( {{{\left| {{\Delta _{Z,h}}u} \right|}^{\frac{{\alpha  - 1}}{p}}}\left( {{\Delta _{Z,h}}u} \right)} \right)\left( {\xi   {e^{h'Z}}} \right)}}{{{{\left| h \right|}^{\frac{{1 + \theta \alpha }}{p}}}}}} \right\|_{{L^p}\left( {{B_{\frac{{R + r}}{2} + {h_0}}}} \right)}^p  \nonumber \\
  \le& c\int_{{B_{\frac{{R + r}}{2} + 2{h_0}}}} {\frac{{{{\left| {{\Delta _{Z,h}}u} \right|}^{\alpha  - 1 + p}}}}{{{{\left| h \right|}^{1 + \theta \alpha }}}}d\xi }  \le c\int_{{B_R}} {\frac{{{{\left| {{\Delta _{Z,h}}u} \right|}^\alpha }}}{{{{\left| h \right|}^{1 + \theta \alpha }}}}d\xi }  \nonumber \\
   \le& c\left( {\int_{{B_R}} {{{\left( {\frac{{\left| {{\Delta _{Z,h}}u} \right|}}{{{{\left| h \right|}^{\frac{{1 + \theta \alpha }}{\alpha }}}}}} \right)}^{\frac{{\alpha q}}{{q - p + 2}}}}d\xi }  + 1} \right),
\end{align}
where $c=c(Q,h_0,p)>0$. Substituting \eqref{eq475} and \eqref{eq476} into \eqref{eq474} yields
\begin{align}\label{eq477}
   & \left\| {\frac{{{\Delta _{Z,h'}}{\Delta _{Z,h}}u}}{{{{\left| {h'} \right|}^{\frac{p}{{\alpha  - 1 + p}}}}{{\left| h \right|}^{\frac{{1 + \theta \alpha }}{{\alpha  - 1 + p}}}}}}} \right\|_{{L^{\alpha  - 1 + p}}\left( {{B_r}} \right)}^{\alpha  - 1 + p}\nonumber \\
  \le &  c\int_{{B_R}} {{{\left| {{\nabla _H}\left( {\frac{{{{\left| {{\Delta _{Z,h}}u} \right|}^{\frac{{\alpha  - 1}}{p}}}\left( {{\Delta _{Z,h}}u} \right)\varphi }}{{{{\left| h \right|}^{\frac{{1 + \theta \alpha }}{p}}}}}} \right)} \right|}^p}d\xi }  + c\left( {\int_{{B_R}} {{{\left( {\frac{{\left| {{\Delta _{Z,h}}u} \right|}}{{{{\left| h \right|}^{\frac{{1 + \theta \alpha }}{\alpha }}}}}} \right)}^{\frac{{\alpha q}}{{q - p + 2}}}}d\xi }  + 1} \right),
\end{align}
where $c=c(Q,h_0,p,\alpha)>0$. We choose $h'=h$ and take the supremum over $h$ for $0 < \left| h \right| <  {{h_0}} $, and then use \eqref{eq472} and \eqref{eq477} to get
\begin{align}\label{eq478}
   & \mathop {\sup }\limits_{0 < \left| h \right| < {h_0}} \int_{{B_r}} {{{\left| {\frac{{\Delta _{Z,h}^2u}}{{{{\left| h \right|}^{\frac{{1 + p + \theta \alpha }}{{\alpha  - 1 + p}}}}}}} \right|}^{\alpha  - 1 + p}}d\xi }  \nonumber \\
   \le&  c(1+\Lambda) \left( \mathop {\sup }\limits_{0 < \left| h \right| < {h_0}}{\int_{{B_R}} {{{\left( {\frac{{\left| {{\Delta _{Z,h}}u} \right|}}{{{{\left| h \right|}^{\frac{{1 + \theta \alpha }}{\alpha }}}}}} \right)}^{\frac{{\alpha q}}{{q - p + 2}}}}d\xi }  + \int_{{B_{R + {4h_0}}}} {{{\left| {{\nabla _H}u} \right|}^q}d\xi } + 1} \right),
\end{align}
where $c=c(Q,h_0,p,q,s,\alpha)>0$.

\textbf{Step 5: Conclusion.} We write
\[\alpha  = q - p + 2,\;\theta  = \frac{{q - p + 1}}{{q - p + 2}}.\]
Then
\[\alpha  - 1 + p = q + 1,\;\frac{{1 + p + \theta \alpha }}{{\alpha  - 1 + p}} = \frac{{q + 2}}{{q + 1}} = \frac{1}{{q + 1}} + 1,\;\frac{{\alpha q}}{{q - p + 2}} = q,\;\frac{{1 + \theta \alpha }}{\alpha } = 1.\]
Thus, \eqref{eq478} becomes
\begin{align}\label{eq479}
   & \mathop {\sup }\limits_{0 < \left| h \right| < {h_0}} \int_{{B_r}} {{{\left| {\frac{{\Delta _{Z,h}^2u}}{{{{\left| h \right|}^{\frac{1}{{q + 1}} + 1}}}}} \right|}^{q + 1}}d\xi }  \nonumber\\
  \le & c(1+\Lambda) \left( \mathop {\sup }\limits_{0 < \left| h \right| < {h_0}} \int_{{B_R}} {{\left( {\frac{{\left| {{\Delta _{Z,h}}u} \right|}}{{\left| h \right|}}} \right)}^q}d\xi   + {\int_{{B_{R + {4h_0}}}} {{{\left| {{\nabla _H}u} \right|}^q}d\xi } + 1} \right),
\end{align}
where $c=c(Q,h_0,p,q,s)>0$. Moreover, we use Lemma \ref{Le22}, $r = R - 4{h_0}$ and \eqref{eq479} to get
\begin{equation}\label{eq480}
\mathop {\sup }\limits_{0 < \left| h \right| < {h_0}} \int_{{B_{R - {4h_0}}}} {{{\left| {\frac{{\Delta _{Z,h}^2u}}{{{{\left| h \right|}^{\frac{1}{{q + 1}} + 1}}}}} \right|}^{q + 1}}d\xi }  \le c(1+\Lambda) \left(  {\int_{{B_{R + {4h_0}}}} {{{\left| {{\nabla _H}u} \right|}^q}d\xi } + 1} \right),
\end{equation}
where $c=c(Q,h_0,p,q,s)>0$.
\end{proof}

\section{Proofs of Theorem \ref{Th13} and Theorem \ref{Th14}}\label{Section 4}

In this section, we present the proofs of Theorem \ref{Th13} and Theorem \ref{Th14}. Prior to proving the theorems, we first establish the following proposition regarding the gradient integrability of weak solutions by using the Caccioppoli-type inequality.

\begin{proposition}\label{Pro42}
Let $2\le p <\infty,\;0<s<1$ and $0\le \Lambda \le 1$. If $u \in HW_{{\rm{loc}}}^{1,p}\left( {{B_2}\left( {{\xi _0}} \right) \cap L_{sp}^{p - 1}\left( {{\mathbb{H}^n}} \right)} \right)$ is a weak solution of \eqref{eq0} in ${{B_{2}}\left( {{\xi _0}} \right)}$ satisfying \eqref{eq41}, then
\begin{equation*}
\int_{{B_{\frac{7}{8}}}\left( {{\xi _0}} \right)} {{{\left| {{\nabla _H}u} \right|}^p}d\xi }  \le C\left( {Q,p,s} \right).
\end{equation*}
\end{proposition}

\begin{proof}
Without loss of generality, we assume $\xi_0=0$. We only give the proof for $\omega  = {u_ + }$, the proof of $\omega  = {u_ - }$ is similar. By using Lemma \ref{Le211} with $r=1$, $\xi_0=0$ and $\psi  \in C_0^\infty \left( {{B_{\frac{8}{9}}}} \right)$ such that $\psi=1$ on ${{B_{\frac{7}{8}}}}$, $0 \le \psi \le 1$ and $\left| {{\nabla _H}\psi } \right| \le c\left( Q \right)$ for some $c\left( Q \right)>0$, we get
\begin{align*}
   \int_{{B_{\frac{7}{8}}}} {{{\left| {{\nabla _H}\omega } \right|}^p}d\xi }  \le& c\left( {Q,p} \right)\left( {\int_{{B_1}} {{\omega ^p}d\xi }  + \int_{{B_1}} {\int_{{B_1}} {\frac{{\left( {{{\left| {\omega \left( \xi  \right)} \right|}^p} + {{\left| {\omega \left( \eta  \right)} \right|}^p}} \right)}}{{\left\| {{\eta ^{ - 1}} \circ \xi } \right\|_{{{\mathbb{H}}^n}}^{Q + \left( {s - 1} \right)p}}}d\xi d\eta } } } \right) \\
   &  + c\left( {Q,p} \right){\int _{{{\mathbb{H}}^n}\backslash {B_1}}}\frac{{{{\left| {\omega \left( \eta  \right)} \right|}^{p - 1}}}}{{\left\| {{\eta ^{ - 1}} \circ \xi } \right\|_{{{\mathbb{H}}^n}}^{Q + sp}}}d\eta \cdot\int_{{B_1}} {\left| \omega  \right|d\xi } \\
    \le& c\left( {Q,p} \right)\left( {1 + c\left( {Q,p,s} \right) + c\left( {Q,p,s} \right)} \right)\\
     \le &c\left( {Q,p,s} \right).
\end{align*}
\end{proof}

\textbf{Proof of Theorem \ref{Th13}} By Lemma \ref{Le212}, we know $u \in L_{loc}^\infty \left( \Omega  \right)$. Without loss of generality, we assume $\xi_0=0$, and write
\[{M_R} = {\left\| u \right\|_{{L^\infty }\left( {{B_R}} \right)}} + {\rm{Tai}}{{\rm{l}}_{p - 1,sp,p}}\left( {u;0,R} \right) > 0\]
and
\begin{equation}\label{eq484}
 {u_R}\left( \xi  \right) = \frac{1}{{{M_R}}}u\left( {R\xi } \right)\;\;\rm{for}\;\xi \in B_2.
\end{equation}
By applying a scaling transformation, it suffices to prove
\[{\left\| u _R\right\|_{{C^\gamma }\left( {{B_{\frac{1}{2}}}} \right)}} \le C\left( {Q,p,s,\gamma } \right),\]
from which Theorem \ref{Th13} follows.

Note that $u_R$ is a local weak solution of $ - {\Delta _p}u + \Lambda {R^{p - sp}}{\left( { - {\Delta _p}} \right)^s}u = 0$ in $B_2$ and satisfies
\begin{equation}\label{eq485}
{\left\| {{u_R}} \right\|_{{L^\infty }\left( {{B_1}} \right)}} \le 1,\;{\int _{{\mathbb{H}^n}\backslash {B_1}}}\frac{{{{\left| {{u_R}\left( \eta  \right)} \right|}^{p - 1}}}}{{\left\| \eta  \right\|_{{\mathbb{H}^n}}^{Q + sp}}}d\eta  \le 1,\;\int_{{B_{\frac{7}{8}}}} {{{\left| {{\nabla _H}{u_R}} \right|}^p}d\xi }  \le C\left( {Q,p,s} \right).
\end{equation}

For simplicity, we denote $u_R$ by $u$. Fix $0<\gamma<1$ and choose $i_\infty \in \mathbb{N}\setminus {0}$ such that
\begin{equation}\label{eq492}
1 - \gamma  > \frac{Q}{{p + {i_\infty }}}.
\end{equation}
For $i = 0, \cdots ,{i_\infty }$, we define
\[{q_i} = p + i,\]
and
\[{h_0} = \frac{1}{{112{i_\infty }}},\;{R_i} = \frac{7}{8} - 4{h_0} - 14{h_0}i,\]
so
\[{R_0} + 4{h_0} = \frac{7}{8},\;{R_{{i_\infty }}} + 4{h_0} = \frac{3}{4},\;4{h_0} <R_i<1-5{h_0} .\]
By using Proposition \ref{Pro41} with
\[R = {R_i},q = {q_i},\;i = 0, \cdots ,{i_\infty },\]
and we have by \eqref{eq485}
\begin{equation}\label{eq486}
\mathop {\sup }\limits_{0 < \left| h \right| < {h_0}} \int_{{B_{R_0 - {4h_0}}}} {{{\left| {\frac{{\Delta _{Z,h}^2u}}{{{{\left| h \right|}^{\frac{1}{{{q_1}}} + 1}}}}} \right|}^{{q_1}}}d\xi }  \le c(1 + \Lambda )\left( {\int_{{B_{\frac{7}{8}}}} {{{\left| {{\nabla _H}u} \right|}^p}d\xi }  + 1} \right) \le C\left( {Q,p,s,\gamma} \right).
\end{equation}
Consider ${R_i} - 10{h_0} = {R_{i + 1}}   + 4{h_0}$ for $i = 0, \cdots ,{i_\infty }$, we use Lemma \ref{Le213} in \eqref{eq486} to get
\begin{equation}\label{eq487}
\int_{{B_{{R_1} + 4{h_0}}}} {{{\left| {{\nabla _H}{u}} \right|}^{{q_1}}}d\xi }  \le C\left( {Q,p,s,\gamma} \right).
\end{equation}
Applying Proposition \ref{Pro41} and \eqref{eq487} once again, we obtain
\begin{equation}\label{eq488}
\mathop {\sup }\limits_{0 < \left| h \right| < {h_0}} \int_{{B_{{R_1} - 4{h_0}}}} {{{\left| {\frac{{\Delta _{Z,h}^2u}}{{{{\left| h \right|}^{\frac{1}{{{q_2}}} + 1}}}}} \right|}^{{q_2}}}d\xi }  \le c(1 + \Lambda )\left( {\int_{{B_{{{{R_1} + 4{h_0}}}}}} {{{\left| {{\nabla _H}u} \right|}^{{q_1}}}d\xi }  + 1} \right) \le C\left( {Q,p,s,\gamma} \right).
\end{equation}
Moreover, by Lemma \ref{Le213} in \eqref{eq488}, we have
\begin{equation}\label{eq489}
\int_{{B_{{R_2} + 4{h_0}}}} {{{\left| {{\nabla _H}{u}} \right|}^{{q_2}}}d\xi }  \le C\left( {Q,p,s,\gamma} \right).
\end{equation}
Repeating this procedure, it yields
\begin{equation}\label{eq490}
  \int_{{B_{{R_{i+1}} + 4{h_0}}}} {{{\left| {{\nabla _H}{u}} \right|}^{{q_{i+1}}}}d\xi }  \le C\left( {Q,p,s,\gamma} \right)
\end{equation}
for all $i = 0, \cdots ,{i_\infty }-1.$ Taking $i={i_\infty }-1$ in \eqref{eq490} and using ${\left\| {{u}} \right\|_{{L^\infty }\left( {{B_1}} \right)}} \le 1$, we deduce
\begin{equation}\label{eq491}
{\left\| u \right\|_{H{W^{1,{q_{{i_\infty }}}}}\left( {{B_{{R_{{i_\infty }}} + 4{h_0}}}} \right)}} \le C\left( {Q,p,s,\gamma } \right).
\end{equation}
By means of \eqref{eq492}, we know ${q_{{i_\infty }}}>Q$, so we use ${R_{{i_\infty }}} + 4{h_0} = \frac{3}{4}$ and Lemma \ref{Le23} to get
$u \in C_{loc}^\gamma \left( {{B_{\frac{3}{4}}}} \right)$. In fact, we take a cut-off function $\psi$ between ${{B_{\frac{1}{2}}}}$ and ${{B_{\frac{3}{4}}}}$. Then
\begin{align}\label{eq493}
{\left\| u \right\|_{{C^\gamma }\left( {{B_{\frac{1}{2}}}} \right)}}& \le {\left\| {u\psi } \right\|_{{C^\gamma }\left( {{{\mathbb{H}}^n}} \right)}} \le c{\left\| {u\psi } \right\|_{H{W^{1,{q_{{i_\infty }}}}}\left( {{{\mathbb{H}}^n}} \right)}} = c{\left\| {u\psi } \right\|_{H{W^{1,{q_{{i_\infty }}}}}\left( {{B_{\frac{3}{4}}}} \right)}}\nonumber\\
& \le c{\left\| u \right\|_{H{W^{1,{q_{{i_\infty }}}}}\left( {{B_{\frac{3}{4}}}} \right)}} \le C\left( {Q,p,s,\gamma } \right).
\end{align}
Therefore, Theorem is proved.

\textbf{Proof of Theorem \ref{Th14}} After rescaling the variables as in the proof of Theorem \ref{Th13}, it suffices to show that
$${\left\| u_R \right\|_{{C^{1,\alpha }}\left( {{B_{\frac{1}{8}}}} \right)}} \le c\left( {Q,p,s,\gamma } \right),$$
where the function $u_R$ defined in \eqref{eq484} satisfies the following conditions:
\begin{equation}\label{eq52}
{\left\| {{u_R}} \right\|_{{L^\infty }\left( {{B_1}} \right)}} \le 1,\;{\int _{{\mathbb{H}^n}\backslash {B_1}}}\frac{{{{\left| {{u_R}\left( \eta  \right)} \right|}^{p - 1}}}}{{\left\| \eta  \right\|_{{\mathbb{H}^n}}^{Q + sp}}}d\eta  \le 1.
\end{equation}
By Theorem \ref{Th13}, there exists $\gamma>\frac{sp}{p-1}$ such that
\[\mathop {\sup }\limits_{\xi  \ne \eta  \in {B_{\frac{1}{2}}}\left( {{\xi _0}} \right)} \frac{{\left| {{u_R}\left( \xi  \right) - {u_R}\left( \eta  \right)} \right|}}{{{{\left\| {{\eta ^{ - 1}} \circ \xi } \right\|}_{{\mathbb{H}^n}}^\gamma }}} \le c\left( {Q,p,s,\gamma } \right).\]
Now take any $\xi_0 \in {B_{\frac{1}{4}}}$, then we deduce from the above formula and the choice of $\gamma$ that
\[\int_{{B_{\frac{1}{4}}}\left( {{\xi _0}} \right)} {\frac{{{{\left| {{u_R}\left( {{\xi _0}} \right) - {u_R}\left( \eta  \right)} \right|}^{p - 1}}}}{{{{\left\| {{\eta ^{ - 1}} \circ {\xi _0}} \right\|}_{{\mathbb{H}^n}}^{Q + sp}}}}d\eta }  \le c\left( {Q,p,s,\gamma } \right)\int_{{B_{\frac{1}{4}}}\left( {{\xi _0}} \right)} {\frac{1}{{{{\left\| {{\eta ^{ - 1}} \circ {\xi _0}} \right\|}_{{\mathbb{H}^n}}^{Q + sp - \gamma \left( {p - 1} \right)}}}}d\eta }  = c\left( {Q,p,s,\gamma } \right).\]
Moreover, we have by \eqref{eq52}
\[\int_{{\mathbb{H}^n}\backslash {B_{\frac{1}{4}}}\left( {{\xi _0}} \right)} {\frac{{{{\left| {{u_R}\left( {{\xi _0}} \right) - {u_R}\left( \eta  \right)} \right|}^{p - 1}}}}{{{{\left\| {{\eta ^{ - 1}} \circ {\xi _0}} \right\|}_{{\mathbb{H}^n}}^{Q + sp}}}}d\eta }  \le c\left( {Q,p,s} \right).\]
Hence
\[{\left\| {{{{{\left( { - \Delta _{{{\mathbb{H}},p}}} \right)}^s}}}{u_R}} \right\|_{{L^\infty }\left( {{B_{\frac{1}{4}}}} \right)}} \le c\left( {Q,p,s,\gamma } \right),\]
and also
\[{\left\| {{{{\left( { - \Delta _{{{\mathbb{H}},p}}} \right)}}}{u_R}} \right\|_{{L^\infty }\left( {{B_{\frac{1}{4}}}} \right)}} \le c\left( {Q,p,s,\gamma } \right),\]
which together with \eqref{eq52} and Theorem 1.2 in \cite{MZ21} imply
\[{\left\| {{u_R}} \right\|_{{C^{1,,\alpha }}\left( {{B_{\frac{1}{8}}}} \right)}} \le c\left( {Q,p,s,\gamma } \right).\]
%\section{H\"{o}lder Continuity}\label{Section 4}

%\section{Harnack Inequality}\label{Section 5}

%\section{Weak Harnack Inequality}\label{Section 6}

\section*{Acknowledgements}
The author would like to thank Pengcheng Niu for comments and suggestions. This work was supported by the National Natural Science Foundation of China (No. 12501269) and the Scientific Research Program Funded by Shaanxi Provincial Education Department (No. 25JK0360).

\section*{Declarations}
\subsection*{Conflict of interest} The authors declare that there is no conflict of interest. We also declare that this
manuscript has no associated data.

\subsection*{Data Availability} Data sharing is not applicable to this article as no datasets were generated or analysed during the current study.

%{\small

\end{document}